\documentclass[a4paper,UKenglish,cleveref, autoref, thm-restate]{lipics-v2021}
\hideLIPIcs  %uncomment to remove references to LIPIcs series (logo, DOI, ...), e.g. when preparing a pre-final version to be uploaded to arXiv or another public repository

\title{On Blockers and Transversals of Maximum Independent Sets in Co-Comparability Graphs}

\titlerunning{Blockers and Transversals of Independent Sets in Co-Comparability Graphs} %TODO optional, please use if title is longer than one line

\author{Felicia Lucke}{Department of Informatics, University of Fribourg, Switzerland}{felicia.lucke@unifr.ch}{https://orcid.org/0000-0002-9860-2928}{}
\author{Bernard Ries}{Department of Informatics, University of Fribourg, Switzerland}{bernard.ries@unifr.ch}{https://orcid.org/0000-0003-4395-5547}{}%TODO mandatory, please use full name; only 1 author per \author macro; first two parameters are mandatory, other parameters can be empty. Please provide at least the name of the affiliation and the country. The full address is optional. Use additional curly braces to indicate the correct name splitting when the last name consists of multiple name parts.

\authorrunning{F. Lucke and B. Ries} 

 %TODO mandatory: Please choose ACM 2012 classifications from https://dl.acm.org/ccs/ccs_flat.cfm 

 % mandatory; please add comma-separated list of keywords

\nolinenumbers %uncomment to disable line numbering

\newcommand\yes{\textsc{Yes}}
\newcommand\np{\textsc{NP}}
\newcommand{\trans}{\textsc{Trans\-versal}}
\newcommand{\del}{\textsc{Deletion Blocker}}
\newcommand{\set}[1]{\ensuremath{ \left\lbrace #1 \right\rbrace }}
\newtheorem{property}[theorem]{Property}
 \newtheorem{claim1}{Claim}[theorem]
\DeclareMathOperator*{\pos}{pos}

\newcommand{\mincut}{\textsc{Vertex Cut}}
\DeclareMathOperator*{\leftext}{leftext}
\DeclareMathOperator*{\rightext}{rightext}
\newcommand\I{\mathcal{I}}
\DeclareMathOperator*{\level}{level}

\newcommand{\intInterval}[1]{[\![#1]\!]}

\usepackage{graphicx}
\usepackage[dvipsnames]{xcolor}
\usepackage{tikz}
\usetikzlibrary{decorations.pathreplacing}
\usepackage{tikz-layers}
\usepackage{algorithm}
\usepackage{algpseudocode}
\usepackage{cite}

\definecolor{nicered}{RGB}{204,0,0}
\definecolor{lightblue}{RGB}{153,204,255}

\tikzstyle{vertex}=[thin,circle,inner sep=0.cm, minimum size=1.7mm, fill=black, draw=black]
\tikzstyle{evertex}=[thin,circle,inner sep=0.cm, minimum size=1.7mm, fill=none, draw=black]
\tikzstyle{bvertex}=[thin,circle,inner sep=0.cm, minimum size=1.7mm, fill=lightblue, draw=lightblue]
\tikzstyle{rvertex}=[thin,circle,inner sep=0.cm, minimum size=1.7mm, fill=nicered,draw=nicered]
\tikzstyle{edge}=[thick, draw = gray]
\tikzstyle{rededge}=[thick, draw = nicered]
\tikzstyle{bluedge}=[thick, draw = lightblue]
\tikzstyle{arc} = [thick, draw = Gray, >=stealth, shorten > = 0pt, ->]
\tikzstyle{bluearc} = [thick, draw = lightblue, >=stealth, shorten > = 0pt, ->]
\tikzstyle{redarc} = [thick, draw = nicered, >=stealth, shorten > = 0pt, ->]
\tikzstyle{thinedge}=[thin, draw = black]

\begin{document}

\maketitle

\begin{abstract}
In this paper, we consider the following two problems: (i) \del($\alpha$) where we are given an undirected graph $G=(V,E)$ and two integers $k,d\geq 1$ and ask whether there exists a subset of vertices $S\subseteq V$ with $|S|\leq k$ such that $\alpha(G-S) \leq \alpha(G)-d$, that is the independence number of $G$ decreases by at least $d$ after having removed the vertices from $S$; (ii) \trans($\alpha$) where we are given an undirected graph $G=(V,E)$ and two integers $k,d\geq 1$ and ask whether there exists a subset of vertices $S\subseteq V$ with $|S|\leq k$ such that for every maximum independent set~$I$ we have $|I\cap S| \geq d$. We show that both problems are polynomial-time solvable in the class of co-comparability graphs by reducing them to the well-known \mincut\ problem. Our results generalize a result of [Chang et al., Maximum clique transversals, Lecture Notes in Computer Science 2204, pp.~32-43, WG 2001] and a recent result of [Hoang et al., Assistance and interdiction problems on interval graphs, Discrete Applied Mathematics 340, pp.~153-170, 2023].
%For a given family of vertex sets $\mathcal{C}$, a $d$-transversal of~$\mathcal{C}$ is a vertex set $S$ intersecting every set
%in~$\mathcal{C}$ in at least $d$ elements. Given a graph parameter $\pi$, a $d$-deletion blocker is a vertex set $S$ whose deletion reduces the graph parameter by $d$. We consider $d$-transversals and $d$-deletion blockers of maximum independent sets, where we want
%the set $S$ to be of minimum size. We show
%how to represent independent sets in co-comparability graphs as paths in a directed acyclic graph. By joining multiple
%copies of this directed acyclic graph we construct a graph on which a minimum vertex cut corresponds to a
%$d$-transversal ($d$-deletion blocker, resp.) of the original graph. 
\end{abstract}

%=============================================================================================

\section{Introduction}
\label{sec:intro}

Graph parameters like for instance the independence number, the clique number, the chromatic number and the domination number have been intensively studied in the literature. While one is usually interested in maximising or minimising such parameters, another interesting question one may ask is by how much one can decrease the value of a graph parameter by using a limited number of some predefined graph operations (like for instance vertex deletions or edge contractions). This leads to so-called {\sc blocker} problems which are defined as follows. For a fixed set $S$ of graph operations, a given graph $G=(V,E)$, two integers $k$ and $d$, and some graph parameter $\pi$, we want to know if we can transform $G$ into a graph $G'$ by using at most $k$ operations from the set $S$ and such that $\pi(G')\leq \pi(G)-d$. The integer $d$ is called the \emph{threshold}. 

 Over the last years, blocker problems have been investigated intensively in the literature (see for instance \cite{BBPR15,BTT11,BCPRW12,BCPRW11,CL20,CPW11,DPPR18, GLR21,GLMR23,GMR21a,GMR21b, HLW23,LM23,PRS94,PBP14,PPR16,PPR17,RCPWCZ10,ZRPWCB09}) and have relations to many other well known problems (like for instance {\sc Hadwiger Number}, {\sc Bipartite Contraction} and {\sc Maximum Induced Bipartite Subgraph}; see~\cite{DPPR18} for more examples). The graph parameters that have been considered were for instance the matching number (see~\cite{ZRPWCB09}), the chromatic number (see~\cite{PPR16, PPR17}), the (total or semitotal) domination number (see~\cite{GLR21}), the length of a longest path (see~\cite{PRS94}), the clique number (see~\cite{PPR16}), the weight of a minimum dominating set (see~\cite{PWBP15}), the vertex cover number (see~\cite{BTT11}) and the independence number (see~\cite{CPW11, DPPR18}). Regarding the graph operations, the set $S$ always consisted in a single operation: vertex deletion, edge deletion, edge contraction or edge addition. 
%BR:check which references treat which problem and update this paragraph.

A related problem to the blocker problem is the so-called {\sc transversal} problem. Given a graph $G=(V,E)$, a property $\mathcal{P}$ and two integers $d$ and $k$, we want to know if there exists a set $V'\subseteq V$ (resp.\ a set $E'\subseteq E$) of size at most $k$ that intersects each set of vertices (resp.\ set of edges) satisfying property $\mathcal{P}$ on at least $d$ vertices (resp.\ $d$ edges). 
For example, if $\mathcal{P}$ corresponds to ``being a maximum independent set'', the transversal problem consists in asking whether one can find a set of at most~$k$ vertices which intersects every maximum independent set on at least $d$ vertices.  Another example is the well-known {\sc Feedback Vertex Set} problem. 
Here, we are interested in finding a subset of vertices of size at most~$k$ which intersects every cycle on at least one vertex. Thus, it is a transversal problem with $\mathcal{P}$ corresponding to ``being a cycle'' and $d=1$.

Transversal problems have also received much attention in the literature over the last years. %(see for instance~\cite{XX}). 
Properties considered were for instance ``being a maximum independent set'' (see~\cite{BCPRW12, BCPRW11}), ``being a maximum matching'' (see~\cite{RCPWCZ10,ZRPWCB09}), ``being a maximal clique'' (see~\cite{LC06}) and ``being a cycle'' (see~\cite{BB02}). %BR:check for references and update this paragraph.

In this paper, we will focus on the blocker problem with $\pi$ being the independence number~$\alpha$ and $S$ consisting in a single operation, namely vertex deletion, as well as on the transversal problem with $\mathcal{P}$ corresponding to ``being a maximum independent set''. We formally define our problems as follows.
\begin{center}
\fbox{
\begin{minipage}{5.5in}
\del($\alpha$)
\begin{description}
\item[Instance:] A graph $G = (V,E)$ and two integers $d,k \geq 1$.
\item[Question:] Is there a set $S \subseteq V$ of cardinality $|S| \leq k$ such that $\alpha(G-S) \leq \alpha(G)-d$?
\end{description}
\end{minipage}}
\end{center}

We denote by $d$-\del($\alpha$) the problem \del($\alpha$), when $d$ is fixed. A \emph{$d$-deletion blocker} of $G = (V,E)$ of size $k$ is a set $ S \subseteq V$ with $|S| \leq k$ such that $\alpha(G-S) \leq \alpha(G)-d$. 

\begin{center}
\fbox{
\begin{minipage}{5.5in}
\trans($\alpha$)
\begin{description}
\item[Instance:] A graph $G = (V,E)$ and two integers $d,k \geq 1$.
\item[Question:] Is there a set $S \subseteq V$ of cardinality $|S| \leq k$ such that for every maximum independent set $I$ we have $|I\cap S| \geq d$?
\end{description}
\end{minipage}}
\end{center}

We denote by $d$-\trans($\alpha$) the problem \trans($\alpha$), when $d$ is fixed. A \emph{$d$-transversal} of $G = (V,E)$ of size $k$ is a set $S \subseteq V$ with $|S| \leq k$ such that for every maximum independent set $I$ we have $|I\cap S| \geq d$. 

Notice that for $d =1$, the problems $1$-\trans($\alpha$) and $1$-\del($\alpha$) are equivalent. This does not hold for $d >1$, as can be seen in Figure~\ref{f-blocker-vs-trans}. If we consider the clique number $\omega$, respectively the property ``being a maximum clique'', we can define in an analogous way the problems \del($\omega$) and \trans($\omega$). Since an independent set in a graph $G$ corresponds to a clique in the complement graph of $G$, it follows immediately that from any computational complexity result for \del($\alpha$) (resp.\ \trans($\alpha$)) in some graph class $\mathcal{G}$, we can deduce a corresponding computational complexity result for \del($\omega$) (resp.\ \trans($\omega$)) in the complement graph class of $\mathcal{G}$ and vice versa.

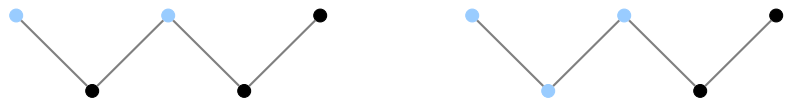
\begin{figure}
\begin{center}
\begin{tikzpicture}
\node[bvertex](v1) at (0,0){};
\node[vertex](v2) at (1,-1){};
\node[bvertex](v3) at (2,0){};
\node[vertex](v4) at (3,-1){};
\node[vertex](v5) at (4,0){};

\draw[edge](v1)--(v2);
\draw[edge](v2)--(v3);
\draw[edge](v3)--(v4);
\draw[edge](v4)--(v5);

\begin{scope}[shift = {(6,0)}]
\node[bvertex](v1) at (0,0){};
\node[bvertex](v2) at (1,-1){};
\node[bvertex](v3) at (2,0){};
\node[vertex](v4) at (3,-1){};
\node[vertex](v5) at (4,0){};

\draw[edge](v1)--(v2);
\draw[edge](v2)--(v3);
\draw[edge](v3)--(v4);
\draw[edge](v4)--(v5);

\end{scope}

\end{tikzpicture}
\caption{\label{f-blocker-vs-trans} A $2$-transversal of the graph $P_5$ (left) and a $2$-deletion blocker of $P_5$ (right), highlighted with light blue vertices. Both are of minimum size. Note that the $2$-deletion blocker is a $2$-transversal but the $2$-transversal is not a $2$-deletion blocker. }
\end{center}
\end{figure}

Since both problems are easily seen to be difficult in general graphs, much effort has been put on special graph classes. For \del($\alpha$), it has been shown that it is \np-complete in split graphs (see~\cite{CPW11}) and thus in chordal and perfect graphs. A recent improvement of this result states that it is \np-complete in chordal and thus in perfect graphs, even if $d= 1$ (see~\cite{LM23}), i.e.\ that $1$-\del($\alpha$) is  \np-complete in chordal graphs. On the positive side, \del($\alpha$) is polynomial time solvable in trees (see~\cite{BTT11, CPW11}), bipartite graphs (see~\cite{BTT11, CPW11}), cobipartite graphs (see~\cite{DPPR18}) and cographs (see~\cite{BTT11}). For split graphs, it has been shown that $d$-\del($\alpha$) is polynomial-time solvable (see~\cite{CPW11}). In a recent paper,  Hoang et al.\ show that it can be solved in polynomial time for interval graphs (see~\cite{HLW23}).

Regarding \trans($\alpha$), it was shown by Bentz et al.\ that this problem is solvable in polynomial time on trees (see~\cite{BCPRW11}) which was then improved to bipartite graphs in~\cite{BCPRW12}. Notice that the authors considered here even weighted maximum independent sets. From the fact that $1$-\trans($\alpha$) and $1$-\del($\alpha$) are equivalent, it follows from~\cite{LM23} that $1$-\trans($\alpha$) is \np-complete for chordal graphs, and thus for perfect graphs. Also, it follows from~\cite{CKL01} that $1$-\trans($\alpha$) is polynomial-time solvable in co-comparability graphs. 

%and maximum clique 1-transversal for comparability graphs with digraphs.
%
%COMMENT: How and where should I include the following two results?
%\cite{CKL01} maximum clique 1-transversal for cographs and maximum clique 1-transversal for comparability graphs with digraphs.
%\cite{Lee11} weighted maximum clique 1-transversal for comparability graphs with digraphs.

In this paper, we will consider the class of co-comparability graphs, which is a subclass of the class of perfect graphs, and show that both problems \trans($\alpha$) and \del($\alpha$) are polynomial-time solvable in this graph class. This generalises the results of~\cite{CKL01} and~\cite{HLW23}. Notice that, as explained above, our results directly imply that \trans($\omega$) and \del($\omega$) can be solved in polynomial time in comparability graphs. In order to show our results, we will reduce both problems to the \mincut\ problem which can be solved in polynomial time (see~\cite{mincut}).

%Other important subclass of perfect graphs are interval, permutation graphs and cographs.
%These are all co-comparability graphs for which we will show in Section~\ref{sec:trans} that the problem is solvable in polynomial time.
%From this result we get immediately that the complementary problem \trans($\omega$), where $\omega$ is the clique number, is solvable in polynomial time on the complementary graph class of comparability graphs.

%For \del($\alpha$) it has been shown that it is \np-hard in split graphs (see~\cite{CPW11}) and thus in chordal and perfect graphs. 
%A recent improvement of this result states that it is \np-hard in chordal and thus in perfect graphs, even if the threshold $d= 1$ (see~\cite{LM23}).
%The problem is polynomial time solvable in trees (see~\cite{BTT11, CPW11}), bipartite graphs (see~\cite{BTT11, CPW11}), cobipartite graphs (see~\cite{DPPR18}), cographs (see~\cite{BTT11}) and also for split graphs if the threshold $d$ is fixed (see~\cite{CPW11}).
%In a recent arxive preprint Hoang, Lendl and Wulf show that it is solvable in polynomial time for interval graphs (see~\cite{HLW23}).
%In this paper, with the approach we use for transversals we show in Section~\ref{sec-blocker} that also the deletion blocker problem is solvable in polynomial time on co-comparability graphs.
%Again this leads to a result for the complementary problem \del($\omega$) for comparability graphs.

Our paper is structured as follows. Section~\ref{sec:prelim} contains notations and terminology. In Section~\ref{sec:indepsettrans}, we present some important properties of maximum independent sets in co-comparability graphs. Section~\ref{sec:trans} contains one of our main results stating that \trans($\alpha$) is solvable in polynomial time in co-comparability graphs. Sections~\ref{sec-structure-nonmax} and~\ref{sec-blocker} deal with the \del($\alpha$) problem. After introducing some more properties of independent sets in co-comparability graphs (Section~\ref{sec-structure-nonmax}), we show our second main result, namely that \del($\alpha$) is polynomial-time solvable in co-comparability graphs in Section~\ref{sec-blocker}. We finish with a conclusion and further research directions in Section~\ref{sec:conclusion}.

%============================================================================================

\section{Preliminairies}
\label{sec:prelim}

In this paper, we only consider finite graphs without self-loops and multiple edges. Unless specified otherwise, all graphs will be undirected and not necessarily connected. 
Let $G = (V,E)$ be a graph.  For $U \subseteq V$, we denote by $G[U]$ the subgraph of $G$ induced by $U$,  i.e.\ the graph with vertex set $U$ and edge set $\set{uv \in E|u,v \in U}$. 
We write $G-U = G[V\setminus U]$. We denote by $\overline{G}$ the \emph{complement} of $G$, that is the graph with vertex set $V$ and edge set $\overline{E} = \set{uv|u,v \in V, uv \notin E}$.
We define $N_G(v)$ as the \emph{neighbourhood} of a vertex $v\in V$ in $G$, i.e.\ the set of vertices $w\in V$ such that $vw\in E$. 

A vertex $v \in V$ is said to be \emph{complete} to some vertex set $U\subseteq V$, if $U \subseteq N_G(v)$.  A \emph{clique} is a set of pairwise adjacent vertices and the \emph{clique number $\omega$} denotes the size of a maximum clique of $G$. We call a vertex $v \in V$ \emph{independent} to some vertex set $U\subseteq V$, if $U\cap N_G(v) = \varnothing$. An \emph{independent set} in $G$ is a set of pairwise non-adjacent vertices of $G$. The \emph{independence number $\alpha$} is the size of a maximum independent set of $G$. For an independent set $S$, another independent set $I$ is called an \emph{extension} of $S$ if $S \subseteq I$. The set $S$ is then called extendable. If $I$ is a maximum independent set, we say it is a \emph{maximum extension} of $S$ and $S$ is \emph{max-extendable}.

We denote by $\intInterval{d}$, with $d \in \mathbb{N}^*$, the set $\{1, 2, \ldots, d\}$.

A \emph{transitive ordering} $\prec$ of $V$ is an ordering of the vertices such that if $u \prec v \prec w$ and $uv, vw \in E$, then $uw \in E$. A \emph{comparability} graph is a graph admitting a transitive ordering of its vertices. The complement of a comparability graph is called a \emph{co-comparability} graph. A transitive ordering of a comparability graph  gives an ordering on its complement with the following property.

\begin{property}
\label{prop-transitive}
Let $G=(V,E)$ be a co-comparability graph and let $\prec$ be a transitive ordering of $V$ in the comparability graph $\overline{G}$. Then, $\prec$ is an ordering of $V$ in $G$ such that if $u \prec v \prec w$ and $uv, vw \notin E$, then $uw \notin E$.
\end{property}

In the rest of the paper, whenever we consider an ordering $\prec$ of the vertices of a co-comparability graph~$G$, we mean an ordering satisfying Property~\ref{prop-transitive}, that is, it is a transitive ordering of its complementary comparability graph. From Property~\ref{prop-transitive}, we can directly deduce the following property.

\begin{property}
\label{prop-transitive2}
Let $G=(V,E)$ be a co-comparability graph and let $\prec$ be a transitive ordering of $V$ in the comparability graph $\overline{G}$. Let $u,v,w \in V$ with $v\prec w$, $uv,vw \notin E$ and $uw \in E$. Then $v \prec u$.
\end{property}

Consider a graph $G=(V,E)$ and an ordering $\prec$ of its vertex set $V$. Let $U \subseteq V$ and $v \in V\setminus U$. We say that $U \prec v$, if $u\prec v$ for every $u \in U$.

Let $G = (V,E)$ be a co-comparability graph with a vertex ordering $\prec$ and let $I$ be an independent set in $G$.  We denote by $\pos_I(v)$, the \emph{position of vertex $v$ in $I$}, that is $\pos_I(v)= | \set{u \in I\vert u \prec v} \vert+1$. A \emph{left extension} (resp.\ \emph{right extension}) of $v\in V$ is an independent set $I\subseteq V$, containing only vertices $u\in V$ with $u \prec v$ (resp.\ $v\prec u$) and such that $v$ is independent to $I$. We make the following easy observations.

\begin{observation}
\label{obs-extensions}
Let $G=(V,E)$ be a co-comparability graph with a vertex ordering $\prec$ and $u,v \in V$ with $u \prec v$. 
\begin{itemize}
\item Let $I_\ell^u$ be a left extension of $u$. Then, $I_\ell^u$ is a left extension of $v$. 
\item Let $I_r^v$ be a right extension of $v$. Then, $I_r^v$ is a right extension of $u$. 
\end{itemize}
\end{observation}

Let $G = (V,A)$ be a directed graph and let $s$, $t\in V$.  We call a directed path from $s$ to $t$ an \emph{$s$-$t$-path}. An \emph{$s$-$t$-cut} in $G$ is a vertex set $C \subseteq V \setminus \set{s,t}$, such that there is no $s$-$t$-path in $G-C$.  We will use the following problem.

\begin{center}
\fbox{
\begin{minipage}{5.5in}
\mincut
\begin{description}
\item[Instance:] A directed graph $G = (V,A)$ with two specified vertices $s,t$ and an integer $k \geq 0$.
\item[Question:] Is there an $s$-$t$-cut $C \subseteq V\setminus \set{s,t}$ with $|C| \leq k$?
\end{description}
\end{minipage}}
\end{center}

%=============================================================================================

\section{Some Properties of Independent Sets in Co-Comparability Graphs}
\label{sec:indepsettrans}

In this section, we will present some structural properties of independent sets in co-comparability graphs. These properties are crucial for our proofs in Section \ref{sec:trans}, where we prove that \trans($\alpha$) is polynomial time solvable in this graph class. 

\begin{lemma}
\label{lemma-fixedposition}
Let $G = (V,E)$ be a co-comparability graph with a vertex ordering $\prec$. Let $I_1$, $I_2$ be two maximum independent sets in $G$.  Let $\alpha = \alpha(G)$ and let ${I_1 = \{u_1,\dots, u_\alpha \}}$  and $I_2 = \{v_1, \dots, v_\alpha \}$, where $u_1\prec \dots \prec u_\alpha$ and $v_1 \prec \dots \prec v_\alpha$.  Assume $\exists\ i, j \in \intInterval{\alpha}$, such that $u_i = v_j$. Then $i = j$.
\end{lemma}

\begin{proof}
Assume for a contradiction that there are $i, j \in  \intInterval{\alpha}$, $ i < j$, such that $u_i = v_j$. 
From Property~\ref{prop-transitive}, it follows that, since there is no edge $vu_i$ with $v \in \set{v_1,\dots, v_{j-1}}$ and no edge $u_iu$ with $u\in \set{u_{i+1}, \dots, u_\alpha}$, there is no edge $vu$ with $v\in \set{v_1,\dots, v_{j-1}}$  and $u \in \set{u_{i+1}, \dots, u_\alpha}$.
Hence, $\set{v_1,\dots, v_{j-1}, u_i, \dots, u_\alpha}$ is an independent set of size at least $\alpha + 1$,  a contradiction.
\end{proof}

So Lemma~\ref{lemma-fixedposition} tells us that in a co-comparability graph $G=(V,E)$, the position of a vertex $v$, which belongs to some maximum independent set, is the same in every maximum independent set containing~$v$. We denote by $\mathcal{I} \subseteq V$ the set of vertices which are contained in some maximum independent set. For any $v \in \mathcal{I}$, we then denote by $\pos(v)$ the position of $v$ in every maximum independent set it belongs to. Thus, for a vertex $v\in \mathcal{I}$, we have $\pos(v) = \pos_I(v)$, for any maximum independent set $I$ containing $v$.

%\begin{definition}
%Let $G = (V,E)$ be a co-comparability graph with a vertex ordering $\prec$.
%Let $I = \set{v_1,\dots, v_\alpha}$ be a maximum independent set of $G$.
%We may assume that $v_1 \prec \dots \prec v_\alpha$.
%Then, we say that vertex $v_i \in I$ with $i\in [\alpha]$ is at \emph{position}~$i$.
%We denote by $\mathcal{I} \subseteq V$ the set of vertices which are contained in some maximum independent set.
%For any $v \in \mathcal{I}$ we denote by $\pos(v)$ the position of $v$ in some maximum independent set. 
%\end{definition}

%We further make the following observation.
%
%\begin{observation}
%\label{obs-ext-indset-positions}
%Let $G = (V,E)$ be a co-comparability graph with a vertex ordering $\prec$. Consider an extendable independent set $I'$ and let $I$ be a maximum independent set containing $I'$..
%\end{observation}

Lemma~\ref{lemma-fixedposition} allows us to partition the vertices in $\mathcal{I}\subseteq V$ into sets $L_1,\dots, L_\alpha \subseteq \mathcal{I}$, where $\alpha = \alpha(G)$, such that for any $v\in \mathcal{I}$ we have $v \in L_p$, $p \in \intInterval{\alpha}$ if and only if there exists an independent set $\set{u_1,\dots, u_{p-1}, u_p = v, u_{p+1},\dots, u_\alpha}$ with $u_1 \prec \dots \prec u_\alpha$.  In other words, $L_p = \set{v \in \mathcal{I}| \pos(v) = p}$, for $p \in \intInterval{\alpha}$ (see also Figure \ref{f-trans-example}). These sets $L_p$ satisfy the following property.

\begin{property} \label{prop-trans-partition-clique}
Let $G = (V,E)$ be a co-comparability graph with a vertex ordering $\prec$.  Let $L_p$, for $p \in \intInterval{\alpha}$, be as defined above. Then, $L_p$ is a clique.
\end{property}

\begin{proof}
Assume for a contradiction that there exist $u,v \in L_p$ with $u \prec v$ such that $uv \notin E$. Let $I_r$ be a right extension of $v$ of maximum size. Let $I_\ell$ be a left extension of $u$ of maximum size. Then $v \prec I_r$ and $I_\ell \prec u$. From Property~\ref{prop-transitive}, we get that $I_\ell \cup \set{u,v} \cup I_r$ is an independent set which has size $(p-1) + 2 + (\alpha-p) = \alpha+1$, a contradiction.
\end{proof}

\begin{figure}[ht]
\begin{centering}
 \begin{tikzpicture} 

\tikzstyle{tedge}=[thin, draw = gray]

\tikzstyle{rahmen1}=[rounded corners = 5pt,draw, dashed, minimum width = 13pt, minimum height = 13pt]
\tikzstyle{rahmen2}=[rounded corners = 5pt,draw, dashed, minimum width = 13pt, minimum height = 40pt]
\tikzstyle{rahmen3}=[rounded corners = 5pt,draw, dashed, minimum width = 13pt, minimum height = 70pt]
\tikzstyle{nodename}=[font=\footnotesize]
          
          \begin{scope}[xscale=1.5]
    \node[bvertex] (a1) at (0,4.5){};
    \node[bvertex] (a2) at (0,3.5){};
    \node[bvertex] (a3) at (1,3){};
    \node[bvertex] (a4) at (1,4){};
    \node[bvertex] (a5) at (1,5){};
    \node[bvertex] (a6) at (2,3.5){};
    \node[bvertex] (a7) at (2,4.5){};
    \node[bvertex] (a8) at (3,4){};
    
    \node[vertex](a9) at (0.25, 3){};
    \node[vertex](a10) at (2.75, 3){};

   % \node[blackvertex] (a11) at (1,2.5){};
    %\node[blackvertex] (a12) at (2.5,2.5){};
    
    \draw[edge] (a1)--(a2);
    \draw[edge] (a1)--(a3);
    \draw[edge] (a1)--(a4);
    \draw[edge] (a3)--(a4);
    \draw[edge] (a3)to [bend left = 25](a5);
    \draw[edge] (a4)--(a5);
    \draw[edge] (a4)--(a6);
    \draw[edge] (a5)--(a6);
    \draw[edge] (a6)--(a7);
    
    \draw[edge] (a9)--(a1);
    \draw[edge] (a9)--(a2);
    \draw[edge] (a9)--(a3);
    \draw[edge] (a9)--(a4);
    \draw[edge] (a9)--(a5);
    
    \draw[edge] (a10)--(a8);
    \draw[edge] (a10)--(a7);
    \draw[edge] (a10)--(a6);
    \draw[edge] (a10)--(a3);

%    \foreach \i in {1,...,8}{
%		  \draw[tedge] (a11)--(a\i);
%		}
%		\foreach \i in {4,5,6,8}{
%		  \draw[tedge] (a12)--(a\i);
%		}
%		\begin{scope}[on background layer]
%    		\draw[tedge] plot [smooth] coordinates {(a12)  (a7)};
%    		\end{scope}

	\node[rahmen2]at (0,4){};	
	\node[rahmen3]at (1,4){};	
	\node[rahmen2]at (2,4){};
	\node[rahmen1]at (3,4){};		

\node[nodename] at (0,5){$L_{1}$};
\node[nodename] at (1,5.5){$L_{2}$};
\node[nodename] at (2,5){$L_{3}$};
\node[nodename] at (3,5){$L_{4}$};
\end{scope}

\end{tikzpicture}
\caption{\label{f-trans-example}A co-comparability graph and the sets $L_1, \dots, L_4$. The black vertices are not contained in any maximum independent set.}
\end{centering}
\end{figure}
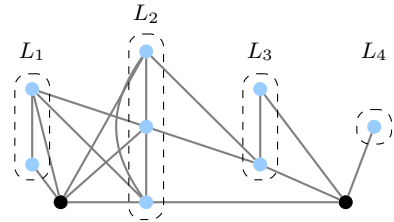

\begin{property}\label{prop-trans-orderingtopos}
Let $G= (V,E)$ be a co-comparability graph with a vertex ordering $\prec$. Let $u,v \in \mathcal{I}$ with $uv \notin E$. Then $\pos(u) < \pos(v)$ if and only if $u \prec v$.
\end{property}
\begin{proof}
Suppose for a contradiction that there are $u,v \in V$ with $uv \notin E, \pos(u) < \pos(v)$ and $v \prec u$. Let $I_\ell^v$ be a maximum left extension of $v$ and $I_\ell^u$ be a maximum left extension of $u$. From Observation~\ref{obs-extensions} we know that $I_\ell^v$ is also a left extension of $u$. Since $\pos(u) < \pos(v) = |I_\ell^v|+1 \leq |I_\ell^u| +1 = \pos(u)$ we get a contradiction.
\end{proof}

Our algorithm that we will present in Section \ref{sec:trans} relies on the sets $L_p$ defined above. Thus, it is important to be able to determine those sets in a co-comparability graph. The following result shows that this can be done in polynomial time.

\begin{lemma} \label{lem-trans-partition-poly}
Let $G = (V,E)$ be a co-comparability graph with a vertex ordering $\prec$. Then, the sets $L_1,\dots, L_\alpha$, where  $\alpha = \alpha(G)$, can be found in polynomial time.
\end{lemma}

\begin{proof}
First notice that it follows from Property \ref{prop-transitive} that if $I_\ell^u$ is a left extension of some vertex $u\in V$ and $I_r^u$ is a right extension of this same vertex $u\in V$, then $I_\ell^u\cup \{u\}\cup I_r^u$ is an independent set in $G$ containing $u$. Furthermore, any independent set containing $u$ consists of a left extension of $u$, a right extension of $u$, and $u$ itself. 
Thus, the size of a maximum independent set in $G$ containing some vertex $u$ can be obtained by determining the maximum left respectively right extension of $u$.  Let us define the two functions $\leftext(\cdot)$ and $\rightext(\cdot)$ as follows: $\leftext(v)$, for $v\in V$, is the size of a maximum left extension of $v$; and $\rightext(v)$, for $v\in V$, is the size of a maximum right extension of $v$.
 From the above, it follows that $\leftext(v) + \rightext(v)+1$ corresponds to the size of a maximum independent set of $G$ which contains $v$. Furthermore, the value $\leftext(v)+1$ gives us the position of $v$ in a maximum independent set $I$ containing $v$, i.e.\ $\pos_I(v)=\leftext(v)+1$. We can also determine $\alpha$, since $\alpha = \max_{v\in V} (\leftext(v) + \rightext(v)+1)$.

From the above, it follows that once we determined $\leftext(v)$ and $\rightext(v)$ for every vertex $v\in V$, we obtain the sets $L_1,\dots, L_\alpha$, since 
\[L_i=\{v\in V|\ \leftext(v)+\rightext(v)+1=\alpha \textit{ and} \; \leftext(v)+1=i\}.\]

Thus, it remains to show how we can compute the functions $\leftext(\cdot)$ and $\rightext(\cdot)$.  Herefore, we use the ordering $\prec$ of the vertices and we get that
\begin{align*}
\leftext(v) &= \max_{u \in V, u \prec v, uv \notin E} (\leftext(u)+1)\\
\rightext(v) & = \max_{u \in V, v\prec u, uv \notin E} (\rightext(u)+1)
\end{align*}

Suppose $V=\{v_1,\ldots,v_n\}$ and $v_1\prec v_2 \prec \cdots \prec v_n$. Clearly, $\leftext(v_1)=\rightext(v_n)=0$. By iterating through the vertices in increasing order with respect to $\prec$, we can calculate the values of $\leftext(v)$, and similarly by iterating in decreasing order we get $\rightext(v)$, for all vertices $v\in V$. Both functions can therefore clearly be computed for all vertices in $\mathcal{O}(|V|^2)$ time. Hence we can find the partition in polynomial time.
\end{proof}

%=============================================================================================

\section{Transversals in Co-Comparability Graphs}
\label{sec:trans}

In this section, we present a polynomial-time algorithm to solve  \trans($\alpha$) in co-comparability graphs. Let $G = (V,E)$ be a co-comparability graph with a vertex ordering~$\prec$, independence number $\alpha = \alpha(G)$ and let $d>0$ be an integer.  
We will construct a directed graph $G' = (V',A')$ such that $(G',k)$ is a \yes-instance of \mincut \ if and only if $(G,d,k)$ is a \yes-instance of \trans($\alpha$) for some integer $k>0$. This equivalence will be shown in Theorem~\ref{theo-transalpha}. Since $G'$ can be constructed in polynomial time (see Theorem~\ref{theo-transalpha}) and since we know that  \mincut \ can be solved in polynomial time (see~\cite{mincut}), we obtain our result.
%From \cite{mincut}, we know that a minimum vertex cut can be found in polynomial time. 
%The size of a minimum vertex cut in $G'$ clearly tells us if $(G',k)$ is a \yes-instance of \mincut.
%Hence, if we can construct $G'$ such that $(G',k)$ is a \yes-instance of \mincut \ if and only if $(G,d,k)$ is a \yes-instance of \trans($\alpha$), we get that \trans($\alpha$) is solvable in polynomial time on co-comparability graphs.

We first make an important observation.

\begin{observation}\label{obs-trans-intersection}
Let $G = (V,E)$ be a co-comparability graph with a vertex ordering $\prec$. Let $S \subseteq V$. Then $S$ is a $d$-transversal of $G$ if and only if it contains at least one vertex from every max-extendable independent set of size $\alpha-d+1$ of $G$.
\end{observation}

Observation~\ref{obs-trans-intersection} tells us that, in order to obtain a $d$-transversal in $G$, we must intersect all max-extendable independent sets of size $\alpha-d+1$ in at least one vertex. 
Therefore, in the following, we will construct a directed graph $G'$ with a source $s$ and a sink $t$, such that every $s$-$t$-path in $G'$ corresponds to a max-extendable independent set of size $\alpha-d+1$ in~$G$. 
This one-to-one correspondence will be proven in Lemma~\ref{indsetsarepaths}. An $s$-$t$-cut in $G'$ will then correspond to a $d$-transversal in $G$ (see Theorem \ref{theo-transalpha}).

Let us now describe the construction of $G'=(V',A')$. Let $\mathcal{I} = \bigcup_{p \in \intInterval{\alpha}} L_p$ be the set of vertices contained in a maximum independent set. 
The vertex set $V'$ of $G'$ consists of $d$~copies $U_1,\dots, U_d$ of $\mathcal{I}$ and two additional vertices $s,t$, that is, $V' = \bigcup_{\ell \in \intInterval{d}} U_\ell \cup \set{s,t}$. 
We denote by $L_{p,\ell}$ the set of vertices in $U_\ell$ that correspond to vertices of $L_p$ in $G$, for $p \in \intInterval{\alpha}$ and $\ell\in \intInterval{d}$, and we say that $\ell$ is the \emph{level} of the vertices in $U_\ell$, denoted by $\level(x) = \ell$ for $ x \in U_\ell$.
 Recall that $\pos(v)$, for $v\in \mathcal{I}$, is the position of $v$ in every maximum independent set it belongs to in $G$. 
For simplicity, we will adopt this same notion for all vertices in $V'\setminus \{s,t\}$, i.e.\ for every vertex $x\in V'\setminus\{s,t\}$ that corresponds to some vertex $v\in L_p$, for $p\in \intInterval{\alpha}$, we will also use $\pos(x)$ and call it the position of $x$ in order to actually refer to the position of $v$, the vertex that $x$ corresponds to in $G$. 
Furthermore, we set $\pos(s) = 0$ and $\level(s) = 1$ as well as $\pos(t) = \alpha+1$ and $\level(t) = d$, in order to simplify the readability of our proofs.

Let $x,y \in V'\setminus \set{s,t}$, where $x \in L_{p,\ell}$ and $y \in L_{p',\ell'}$, with $p, p' \in \intInterval{\alpha}$ and $\ell,\ell' \in \intInterval{d}$. Let $u,v \in \mathcal{I}$, where $x$ corresponds to $u$ and $y$ corresponds to $v$.  We add an arc $(x,y)$, if $\set{u,v}$ is max-extendable in~$G$ and $p' = p+g+1$, $\ell' = \ell+g$ for some integer $g \geq 0$. Finally, for any vertex $x\in L_{p,\ell}$, with $p\in \intInterval{\alpha}$ and $\ell\in \intInterval{d}$,  we add an arc $(s,x)$ if $p = \pos(s) + g+1 = g+1$, $\ell = \level(s) + g = g+1$, for some integer $g \geq 0$, and we add an arc $(x,t)$ if $\pos(t) = \alpha+1 = p + g + 1$, $\level(t) = d = \ell + g$ for some integer $g \geq 0$.

We make the following observations which immediately follow from the definition of $G'=(V',A')$.

\begin{observation}
\label{obs-trans-construction}
Let $G'=(V',A')$ be the directed graph constructed from a co-comparabililty graph $G$ as described above. For any arc $(x,y)\in A'$, we have 
\begin{itemize}
\item[(a)] $\pos(y) > \pos(x)$;
\item[(b)] $\level(y) \geq \level(x)$;
\item[(c)] $\pos(y) - \pos(x) -1= \level(y) - \level(x)$. 
\end{itemize}
\end{observation}

Figures~\ref{f-trans-galpha} and \ref{f-trans-gprime} show the graph $G'$ constructed from the graph $G$ in Figure~\ref{f-trans-example} for $d=1$ (Figure \ref{f-trans-galpha}) and $d=2$ (Figure \ref{f-trans-gprime}). 
%Note that the vertices in the sets $L_{1,2}$ and $L_{4,1}$ in Figure~\ref{f-trans-gprime} are not contained in any $s$-$t$-path and thus could be removed from the graph. More such sets appear when $d > 2$. We keep them to make the definitions and proofs easier.

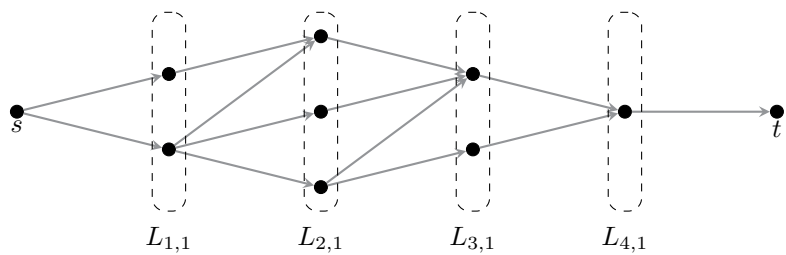
\begin{figure}[ht]
\begin{center}
\centering
  \begin{tikzpicture} 

\tikzstyle{rahmen}=[rounded corners = 5pt,draw, dashed, minimum width = 12.5pt, minimum height = 75pt]
\tikzstyle{nodename}=[font=\normalsize]

	\node[vertex] (s) at (-2,4){};
          
    \node[vertex] (a1) at (0,4.5){};
    \node[vertex] (a2) at (0,3.5){};

    \node[vertex] (a3) at (2,3){};
    \node[vertex] (a4) at (2,4){};
    \node[vertex] (a5) at (2,5){};
    \node[vertex] (a6) at (4,4.5){};
    \node[vertex] (a7) at (4,3.5){};
    
    \node[vertex] (a8) at (6,4){};
    
    \node[vertex] (t) at (8,4){};

    \draw[arc] (a1)--(a5);
    \draw[arc] (a2)--(a5);
    \draw[arc] (a2)--(a4);
    \draw[arc] (a2)--(a3);
    \draw[arc] (a5)--(a6);
    \draw[arc] (a4)--(a6);
    \draw[arc] (a3)--(a6);
    \draw[arc] (a3)--(a7);
    \draw[arc] (a6)--(a8);
    \draw[arc] (a7)--(a8);
    \draw[arc] (s)--(a1);
    \draw[arc] (s)--(a2);
    \draw[arc] (a8)--(t);
  
  \node[rahmen] (l1) at (0,4){};
  \node[rahmen] (l2) at (2,4){};
  \node[rahmen] (l3) at (4,4){};
  \node[rahmen] (l4) at (6,4){};
  
  \node[nodename] at (0,2.3){$L_{1,1}$};
  \node[nodename] at (2,2.3){$L_{2,1}$};
  \node[nodename] at (4,2.3){$L_{3,1}$}; 
  \node[nodename] at (6,2.3){$L_{4,1}$};

\node[nodename, , below] at (-2,4){$s$};
\node[nodename, , below] at (8,4){$t$};

\end{tikzpicture}
\caption{\label{f-trans-galpha}The graph $G'$ constructed from the graph $G$ from Figure \ref{f-trans-example} for $d = 1$. The level of all vertices is 1.}
\end{center}
\end{figure}

\begin{figure}[ht]
\begin{center}
\centering
  \begin{tikzpicture} 
\tikzstyle{vertex}=[thin,circle,inner sep=0.cm, minimum size=1mm, fill=black, draw=black]%overwrite in this figure

\tikzstyle{rahmenb}=[rounded corners = 5pt,draw, dashed, minimum width = 200pt, minimum height = 35pt]

\tikzstyle{vertexsetb}=[rounded corners = 4pt,draw, minimum width = 10pt, minimum height = 30pt]

\node[vertexsetb](L11a) at (2,1){};
\node[vertexsetb](L21a) at (4,1){};
\node[vertexsetb](L31a) at (6,1){};
\node[vertexsetb](L41a) at (8,1){};

\node[vertex](v11) at (2,1.2){};
\node[vertex](v21) at (2,0.8){};
\node[vertex](v31) at (4,0.6){};
\node[vertex](v41) at (4,1){};
\node[vertex](v51) at (4,1.4){};
\node[vertex](v61) at (6,1.2){};
\node[vertex](v71) at (6,0.8){};
\node[vertex](v81) at (8,1){};

\node[vertex](v12) at (2,3.2){};
\node[vertex](v22) at (2,2.8){};
\node[vertex](v32) at (4,2.6){};
\node[vertex](v42) at (4,3){};
\node[vertex](v52) at (4,3.4){};
\node[vertex](v62) at (6,3.2){};
\node[vertex](v72) at (6,2.8){};
\node[vertex](v82) at (8,3){};

\node[vertexsetb](L12a) at (2,3){};
\node[vertexsetb](L22a) at (4,3){};
\node[vertexsetb](L32a) at (6,3){};
\node[vertexsetb](L42a) at (8,3){};

    \node[vertex] (s) at (0,1){};
    \node[vertex] (t) at (10,3){};

    \draw[arc] (s)--(v11);
    \draw[arc] (s)--(v21);
    \draw[arc] (s)--(v32);
    \draw[arc] (s)--(v42);
    \draw[arc] (s)--(v52);

    \draw[arc] (v61)--(t);
    \draw[arc] (v71)--(t);
    \draw[arc] (v82)--(t);
    
    \draw[arc] (v11)--(v51);
    \draw[arc] (v12)--(v52);
    \draw[arc] (v21)--(v31);
    \draw[arc] (v22)--(v32);
    \draw[arc] (v21)--(v41);
    \draw[arc] (v22)--(v42);
    \draw[arc] (v21)--(v51);
    \draw[arc] (v22)--(v52);
    \draw[arc] (v31)--(v61);
    \draw[arc] (v31)--(v71);
    \draw[arc] (v32)--(v72);
    \draw[arc] (v32)--(v62);
    \draw[arc] (v41)--(v61);
    \draw[arc] (v42)--(v62);
    \draw[arc] (v51)--(v61);
    \draw[arc] (v52)--(v62);
    \draw[arc] (v61)--(v81);
    \draw[arc] (v62)--(v82);
    \draw[arc] (v71)--(v81);
    \draw[arc] (v72)--(v82);

    \draw[arc] (v11)--(v62);
    \draw[arc] (v21)--(v62);
    \draw[arc] (v21)--(v72);
    \draw[arc] (v31)--(v82);
    \draw[arc] (v41)--(v82);
    \draw[arc] (v51)--(v82);

\node[font = \small, , below] at (0,1){$s$};
\node[font = \small, , above] at (10,3){$t$};	
\node[font = \scriptsize, below] at (2,0.45){$L_{1,1}$};
\node[font = \scriptsize, below] at (4,0.45){$L_{2,1}$};
\node[font = \scriptsize, below] at (6,0.45){$L_{3,1}$};
\node[font = \scriptsize, below] at (8,0.45){$L_{4,1}$};
\node[font = \scriptsize, above] at (2,3.55){$L_{1,2}$};
\node[font = \scriptsize, above] at (4,3.55){$L_{2,2}$};
\node[font = \scriptsize, above] at (6,3.55){$L_{3,2}$};
\node[font = \scriptsize, above] at (8,3.55){$L_{4,2}$};

\node[font = \small ] at (9.25,3.55){level $2$};
\node[font = \small ] at (9.25,1.5){level $1$};

\node[rahmenb] at (5,1){};
\node[rahmenb] at (5,3){};

\end{tikzpicture}
\caption{\label{f-trans-gprime} The graph $G'$  constructed from the graph $G$ from Figure \ref{f-trans-example} for $d = 2$.  }
\end{center}
\end{figure}
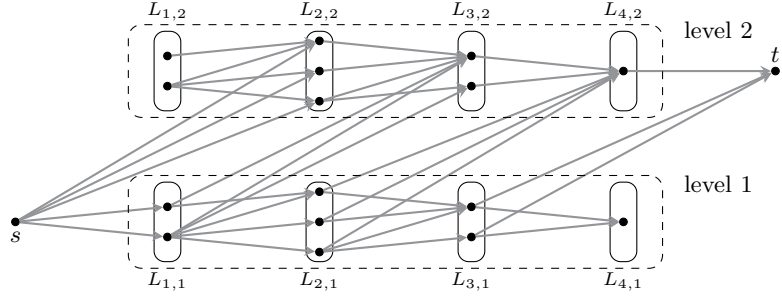

Before we show the one-to-one correspondence between max-extendable independent sets of size $\alpha-d+1$ in $G$ and $s$-$t$-paths in $G'$, we present some useful properties.

\begin{property}
\label{prop-trans-emptypositionspath}
Let $P $ be an $s$-$t$-path in $G'$ with vertices $s,x_1,\dots, x_h, t$ in that order. There exist exactly $d-1$ distinct integers in $\intInterval{\alpha}$, say $g_1,\ldots, g_{d-1}$, such that $\pos(x_i)\neq g_1,\ldots,g_{d-1}$ for all $i\in \intInterval{h}$.
For any other integer $g \in \intInterval{\alpha}\setminus \set{g_1,\dots, g_{d-1}}$,  there exists exactly one vertex $x \in V(P)\setminus \{s,t\}$ such that $\pos (x)=g$. 
\end{property}

\begin{proof}
Let $P $ be an $s$-$t$-path in $G'$ with vertices  $s,x_1,\dots, x_h, t$ in that order. Consider an arc $(x_i,x_{i+1})$ in $P$.  Assume that $x_i \in L_{p,\ell}$ and $x_{i+1}\in L_{p',\ell'}$, with $p,p' \in \intInterval{\alpha}$ and $\ell,\ell'\in \intInterval{d}$. 
From Observation~\ref{obs-trans-construction}(c), it follows that $p' = p+ 1 + \ell'-\ell$ and by Observation~\ref{obs-trans-construction}(b), we know that $\ell'-\ell\geq 0$. 
Hence, we skip $\ell'-\ell$ positions between $x_i$ and $x_{i+1}$, which will not be used by any vertex in $P$, that is, there are $\ell'-\ell$ integers between $p$ and $p'$ which do not correspond to any position of some vertex in $P$. 
Since we start at level~$1$ (recall that $\level(s) = 1$) and we end at level $d$ (recall that $\level(t) = d$), we get that there are exactly $d-1$ distinct integers in $\intInterval{\alpha}$, say $g_1,\ldots, g_{d-1}$, such that $\pos(x_j)\neq g_1,\ldots,g_{d-1}$ for all $j\in \intInterval{h}$. 
Furthermore, since $\pos(x_j') > \pos(x_j)$, for any $j'>j$, with $j,j'\in \intInterval{h}$ (see Observation~\ref{obs-trans-construction}(a)), it is obvious to see that for any other integer $g \in \intInterval{\alpha}\setminus \set{g_1,\dots, g_{d-1}}$, there exists exactly one vertex $x \in V(P)\setminus \set{s,t}$ such that $\pos(x) = g$. 
\end{proof}

The next property gives the exact number of vertices in any $s$-$t$-path in $G'$.

\begin{property}
\label{prop-trans-pathlength}
Every $s$-$t$-path $P$ in $G'$ contains $\alpha-d+3$ vertices.
\end{property}

\begin{proof}
From Property \ref{prop-trans-emptypositionspath}, we know that there are $d-1$ positions that do not correspond to any position of some vertex in $P$ and that all other positions do correspond each to a different position of some vertex in $V(P) \setminus \set{s,t}$.  Since there are in total $\alpha$ possible positions of which none corresponds to the positions of $s$ and $t$ (recall that $\pos(s)=0$ and $\pos(t)=\alpha+1$), there are $\alpha-d+1$ vertices in $V(P)\setminus \set{s,t}$, and hence, $P$ contains exactly $\alpha-d+3$ vertices.
\end{proof}

The following property immediately follows from the definition of the position of a vertex in $\mathcal{I}$ and Observation~\ref{obs-trans-construction}(a).

\begin{property}
\label{prop-trans-emptypositionsindset}
Let $I = \set{v_1,\dots, v_{\alpha-d+1}}$ be a max-extendable independent set in $G$. There exist exactly $d-1$ distinct integers in $\intInterval{\alpha}$, say $g_1,\ldots, g_{d-1}$ such that $\pos(v_i)\neq g_1,\ldots,g_{d-1}$ for all $i\in \intInterval{\alpha-d+1}$.
\end{property}

We are now ready to show the one-to-one correspondence between $s$-$t$-paths in $G'$ and max-extendable independents sets of size $\alpha-d+1$ in $G$.

\begin{lemma}
\label{indsetsarepaths}
Let $G=(V,E)$ be a co-comparability graph with a vertex ordering $\prec$. Let $\alpha = \alpha(G)$, $d> 0$ be an integer and consider $G'=(V',A')$,  constructed as described above. Then, every max-extendable independent set of size $\alpha-d+1$ in $G$ corresponds to an $s$-$t$-path in $G'$ and vice versa.
\end{lemma}

\begin{proof}
Let us first consider an $s$-$t$-path $P= s, x_1,x_2,\dots, x_{h}, t$ in $G'$.  We know from Property~\ref{prop-trans-pathlength} that this path consists of exactly $\alpha-d+3$ vertices, hence $h = \alpha-d+1$.  Let $v_i\in V$ be the vertex in $G$ corresponding to $x_i\in V'$, for $i\in \intInterval{\alpha-d+1}$. Notice that by Observation~\ref{obs-trans-construction}(a), we have $\pos(v_1)<\pos(v_2)<\ldots <\pos(v_{\alpha-d+1})$. 

\begin{claim1} \label{claim-trans-isindset}
$\set{v_1,\ldots,v_{\alpha-d+1}}$ is an independent set in $G$. 
\end{claim1}

\begin{claimproof}
By construction, if there is an arc $(x_i, x_{i+1})$ in $G'$, for $i\in \intInterval{\alpha-d}$, then $v_i$ and $v_{i+1}$ are necessarily non-adjacent. 
%irrelevant
%Hence, it remains to show that for any two vertices $x_i,x_{i+j} \in V(P)$, with $j>1$ and $i,i+j\in \intInterval{\alpha-d+1}$, their corresponding vertices $v_i, v_{i+j}\in V$ are non-adjacent.  By construction, we also know that if there is an arc $(x_i, x_{i+1})$ in $G'$, for $i\in \intInterval{\alpha-d}$, then the set $\{v_i,v_{i+1}\}$ is max-extendable. 
Since $\pos(v_i)<\pos(v_{i+1})$, we have by Property~\ref{prop-trans-orderingtopos} that $v_i \prec v_{i+1}$. 
Thus, we obtain that $v_i \prec v_{i+1} \prec \dots \prec v_{i+j}$. 
Now using Property~\ref{prop-transitive}, we conclude that $v_i, v_{i+j}$ are non-adjacent, for $i, i+j \in \intInterval{\alpha-d+1}$, $j > 1$, such that $x_i, x_{i+j} \in V(P)$.
 Thus, $\set{v_1,\ldots,v_{\alpha-d+1}}$ is an independent set in $G$. 
%From the above we know that $\pos(w_h)<\pos(w_{h+1})<\ldots<\pos(w_{h+j})$. 
%Since in addition we know that the vertices corresponding to $w_h$ and $w_{h+1}$ are non-adjacent and thus, $v_h, v_{h+j}$ are non-adjacent. 
\end{claimproof}

\begin{claim1}
The independent set $\set{v_1, \dots, v_{\alpha-d+1}}$ is max-extendable in $G$. 
\end{claim1}

\begin{claimproof}
As mentioned before, by construction, we know that if there is an arc $(x_i, x_{i+1})$ in~$G'$, then $\set{v_i, v_{i+1}}$ is max-extendable in $G$.  Let now $J'_i \subseteq V$, for $i \in \intInterval{\alpha-d}$, be a maximum extension of $\set{v_i, v_{i+1}}$. Let further $J_i = \set{w \in J_i'| \pos(v_i) < \pos(w) < \pos(v_{i+1})} $ (see Figure~\ref{extensiontoindset}).  
We also adapt this definition to the first vertex $v_1$, and denote by $J_0'$ a maximum extension of $\set{v_1}$  (recall that $v_1\in \mathcal{I}$). 
We choose $J_0 = \set{w \in J_0' | \pos(w) <\pos(v_1)} $. 
Similarly, we denote by $J_{\alpha-d+1}'$ a maximum extension of $\set{v_{\alpha-d+1}}$ and $J_{\alpha-d+1} = \set{w \in J_{\alpha-d+1}'| \pos(v_{\alpha-d+1}) < \pos(w)}$ (recall that $v_{\alpha-d+1}\in \mathcal{I}$). 
The set $\bigcup_{h =0}^{\alpha-d+1} J_h \cup \set{v_1,\dots, v_{\alpha-d+1}}$ is an independent set, since $\set{v_i,v_{i+1}} \cup J_i$ is an independent set, for $i \in \intInterval{\alpha-d+1}$, and we can combine them using Property~\ref{prop-transitive}. 
From the fact that a maximum extension of $\set{v_i,v_{i+1}}$, with $i \in \set{1,\dots, \alpha-d}$, is a maximum independent set $I$, it follows that for any $p\in \intInterval{\alpha}$, there is a vertex in $u \in I$ with $\pos(u) = p$. Thus, we get that $|J_i| = \pos(v_{i+1})-\pos(v_i)-1$, for $i \in \intInterval{\alpha-d}$. Further, we have $ |J_0| = \pos(v_1)-1$ and $|J_{\alpha-d+1}| = \alpha-\pos(v_{\alpha-d+1})$.
Thus, 
\begin{align*}
&\Bigg\vert \set{v_1,\dots, v_{\alpha-d+1}}  \cup \bigcup_{h =0}^{\alpha-d+1} J_h\Bigg\vert \\
&= (\alpha-d+1) +( \pos(v_1)-1 )+ \sum_{h=1}^{\alpha-d} \big( \pos(v_{h+1}) - \pos(v_h)-1\big) + (\alpha-\pos(v_{\alpha-d+1})) \\
&= (\alpha-d+1) + \alpha-(\alpha-d) - 1 = \alpha.
\end{align*}
Hence, it follows that $\set{v_1,\dots, v_{\alpha-d+1}} \cup \bigcup_{h =0}^{ \alpha-d+1} J_h$ is a maximum independent set in $G$, which shows that $\set{v_1,\dots, v_{\alpha-d+1}}$ is max-extendable. 
\end{claimproof}

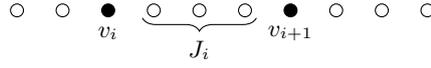
\begin{figure}[t]
\begin{center}
%\documentclass[11pt,a4paper]{article}
%\usepackage[utf8]{inputenc}
%\usepackage{amsmath}
%\usepackage{amsfonts}
%\usepackage{amssymb}
%\usepackage{algorithm}
%\usepackage{amsthm}
%\usepackage{setspace}
%\usepackage[left=3cm,right=3cm,top=3cm,bottom=2cm]{geometry}
%\newtheorem{theorem}{Theorem}
%\newtheorem{lemma}{Lemma}
%\newtheorem{defi}{Definition}
%\usepackage{algpseudocode}
%\usepackage{parskip}
%\usepackage{graphicx}
%
%\newcommand\yes{\textsc{Yes}}
%\newcommand\no{\textsc{No}}
%\newcommand{\set}[1]{\ensuremath{ \left\lbrace #1 \right\rbrace }}
%\newcommand{\del}{\textsc{Vertex Deletion}}
%\newcommand{\cont}{\textsc{Edge Contraction}}
%\newcommand{\mincut}{\textsc{Minimum Vertex Cut}}
%\newcommand{\minedgecut}{\textsc{Minimum Edge Cut}}
%\DeclareMathOperator*{\argmin}{arg\,min}
%\DeclareMathOperator*{\argmax}{arg\,max}
%\DeclareMathOperator*{\pos}{pos}
%\DeclareMathOperator*{\start}{start}
%\DeclareMathOperator*{\dne}{end}
%
%\usepackage{graphicx}
%\usepackage[dvipsnames]{xcolor}
%\usepackage{tikz}
%\usetikzlibrary{arrows,automata}
%\usetikzlibrary{decorations.pathreplacing}
%\usepackage{tikz-layers}
%\begin{document}

\centering
\begin{tikzpicture}
    
	\tikzstyle{interval} = [thick]
	\tikzstyle{intervald} = [thick, dotted]

	\node [evertex](v) at (-1,3){};
	\node [evertex](v) at (-0.4,3){};
	\node [vertex](v) at (0.2,3){};
	\node [evertex](v) at (0.8,3){};
	\node [evertex](v) at (1.4,3){};
	\node [evertex](v) at (2,3){};
	\node [vertex](v) at (2.6,3){};
	\node [evertex](v) at (3.2,3){};
	\node [evertex](v) at (3.8,3){};
	\node [evertex](v) at (4.4,3){};

	\node[font = \small, ] at (0.2,2.7){$v_i$};
	\node[font = \small, ] at (2.6,2.7){$v_{i+1}$};
	
	\draw [decorate,decoration={brace, mirror,amplitude=5pt},xshift=0pt,yshift=0pt] (0.65,2.9) -- (2.15,2.9) node [black,midway,yshift=-5pt, below] {\small $J_i$};

\end{tikzpicture}

%\end{document}
\caption{\label{extensiontoindset} An extension of $v_i$ and $v_{i+1}$ to a maximum independent set. The empty circles represent all vertices of the maximum extension $J_i'$, while the set $J_i \subseteq J_i'$ contains exactly the vertices whose corresponding vertices have positions between those of $v_i$ and $v_{i+1}$. We assume the vertices to be ordered from left to right according to the ordering $\prec$.}
\end{center}
\end{figure}

Let us now prove the converse, i.e.\ that a max-extendable independent set of size $\alpha-d+1$ in $G$ corresponds to an $s$-$t$-path in $G'$. Consider a max-extendable independent set $I = \set{v_1,\dots, v_{\alpha-d+1}}$ of size $\alpha-d+1$ in~$G$. We may assume that $\pos(v_1) < \pos(v_2) < \dots <  \pos(v_{\alpha-d+1})$. 

Let $g_1,\dots, g_{d-1}$ be as in Property \ref{prop-trans-emptypositionsindset}. For $i \in \intInterval{\alpha-d+1}$, let $x_i\in V'$ be the vertex corresponding to $v_i$, with $x_i \in L_{\pos(v_i), \ell}$, where $ \ell= \vert\set{g_k \mid g_k < \pos(v_i), k\in\intInterval{d-1}} \vert + 1$. By Property~\ref{prop-trans-emptypositionsindset}, we know that $1\leq \ell \leq d$, and since $1\leq \pos(v_i) \leq \alpha$, we get that $x_i$ exists. To show the existence of a path $P$  with vertices $s,x_1,\dots, x_{\alpha-d+1},t $ in that order, it remains to show that the arcs $(s,x_1)$, $(x_i, x_{i+1})$, for $i \in\intInterval{\alpha-d}$, and $(x_{\alpha-d+1},t)$ exist.

\begin{claim1}
\label{claim-trans-arcsexist}
The arcs $(s,x_1)$, $(x_i, x_{i+1})$, for $i \in\intInterval{\alpha-d}$, and $(x_{\alpha-d+1},t)$ exist in $G'$.
\end{claim1}

\begin{claimproof}
For the arc $(s,x_1)$, notice that $\vert\set{g_k \mid g_k < \pos(v_1), k\in\intInterval{d-1}} \vert + 1 = \pos(v_1)$, and hence the arc exists by definition. Let $i \in \intInterval{\alpha-d}$.  Let $\ell_i = \vert\set{g_k \mid g_k < \pos(v_i), k\in \intInterval{d-1}} \vert + 1$ be the level of $x_i$ and let $\ell_{i+1} =\vert\set{g_k \mid g_k < \pos(v_{i+1}), k\in \intInterval{d-1}} \vert + 1$ be the level of $x_{i+1}$, as defined above. Recall that by definition the arc $(x_i,x_{i+1})$ exists, if $\pos(x_{i+1}) = \pos(x_i)+g+1$ and $\ell_{i+1} = \ell_i+g$, for some $g\geq 0$, and $\set{v_i, v_{i+1}}$ is max-extendable. It follows from the above that $\ell_{i+1} - \ell_i=\vert\set{g_k \mid \pos(v_i)\leq g_k < \pos(v_{i+1}), k\in \intInterval{d-1}} \vert$, that is, the number of positions between $\pos(v_i)$ and $\pos(v_{i+1})$ that are not used by any vertex in $I$, and hence we get $\pos(x_{i+1}) = \pos(x_i)+\ell_{i+1} -\ell_i+1$. Furthermore, $\set{v_i, v_{i+1}}$ is clearly max-extendable since both belong to $I$. We conclude that the arc $(x_i, x_{i+1})$ necessarily exists. If we consider $x_{\alpha-d+1}$, we get that $\level(x_{\alpha-d+1}) = \vert\set{g_k \mid g_k < \pos(v_{\alpha-d+1}), k\in [d-1]} \vert + 1 = d-1 - (\alpha-\pos(v_{\alpha-d+1}))+1 = \pos(v_{\alpha-d+1}) + d - \alpha$. Hence, the arc $(x_{\alpha-d+1}, t)$ exists by definition. 
\end{claimproof}

It follows that $P = s,x_1,\dots, x_{\alpha-d+1},t $ is a path in~$G'$. This concludes the proof of our lemma.
\end{proof}

Let us show two more properties that we will need in our main theorem of the section. 

\begin{property}
\label{obs-trans-extsets}
Let $G$ be a co-comparability graph with vertex ordering $\prec$ and let $G'$ be the corresponding directed graph, constructed as described above. Let $I$ be a max-extendable independent set of $G$ and let $P$ be the corresponding $s$-$t$-path in $G'$. Consider $v \in I$ and its corresponding vertex $x\in V(P)$. Then, $\pos_{I}(v) = \pos(v) - \level(x)+1$.
\end{property}

\begin{proof}
Let $g_1,\dots, g_{d-1}$ be as in Property \ref{prop-trans-emptypositionsindset}. Recall from the proof of Lemma \ref{indsetsarepaths} that $ \level(x)= \vert\set{g_k \mid g_k < \pos(v_i), k\in\intInterval{d-1}} \vert + 1$.  Hence, the result follows.
\end{proof}

\begin{property}
\label{prop-trans-levelpos}
Let $G$ be a co-comparability graph with vertex ordering $\prec$ and let $G'$ be the corresponding directed graph, constructed as described above. Let $I_1,I_2$ be two max-extendable independent sets of size $\alpha-d+1$ in $G$, and let $P_1,P_2$ be their corresponding paths in $G'$.
Let $v \in I_1 \cap I_2$ such that the corresponding vertices $x_1\in P_1,  x_2\in P_2$ are different.  Assume without loss of generality that $\level(x_1) < \level(x_2)$. Then,  $\pos_{I_1}(v)> \pos_{I_2}(v)$.
\end{property}

\begin{proof}
Let $g_1^1,\dots, g_{d-1}^1 \in \intInterval{\alpha}$ be the integers, such that $\pos(y) \neq g^1_1,\dots, g^1_{d-1}$ for any $y \in P_1$, which exist by Property~\ref{prop-trans-emptypositionsindset}. Let $g_1^2, \dots, g_{d-1}^2$ be this set of integers corresponding to $P_2$.
Recall from the proof of Lemma \ref{indsetsarepaths} that $\level(x_1) = \big\vert \set{g_k^1 \mid g_k^1 < \pos(v)|k \in \intInterval{d-1}} \big\vert+1$ and $\level(x_2) = \big\vert\set{g_k^2 \mid g_k^2 < \pos(v)|k \in \intInterval{d-1}}\big\vert+1$.
Since, by Property~\ref{obs-trans-extsets} $\pos_{I_1}(v) = \pos(v) - \level(x_1)+1$ and $\pos_{I_2}(v) = \pos(v) - \level(x_2)+1$, and since we assume that $\level(x_1) < \level(x_2)$, it follows that  $\pos_{I_1}(v) > \pos_{I_2}(v)$. 
\end{proof}

\begin{theorem}
\label{theo-transalpha}
\trans$(\alpha)$ is polynomial-time solvable for co-comparability graphs.
\end{theorem}

\begin{proof}
Let $G=(V,E)$ be a co-comparability graph and let $(G,d,k)$ be an instance of \trans($\alpha$). We construct the graph $G'=(V',A')$ as described above. We will show that $(G,d,k)$ is a \yes-instance of \trans($\alpha$) if and only if $(G',k)$ is a \yes-instance of \mincut. Let $(G',k)$ be a \yes-instance of \mincut \ and let $C$ be an $s$-$t$-cut of~$G'$ of size at most~$k$. We want to prove that $(G,d,k)$ is a \yes-instance of \trans($\alpha$).

For every vertex in the cut $C\subseteq V'$, we add the corresponding vertex in $G$ to a set $S$. 
We assume for a contradiction that there is an independent set $I$ in $G- S$ of size $\alpha-d+1$ which is max-extendable in $G$.  By Lemma \ref{indsetsarepaths}, we know that there is a path $P$ from $s$ to $t$ in~$G'$ representing $I$.  Since $I \subseteq V\setminus S$, we get that $P\cap C = \varnothing$ and hence, we can find an $s$-$t$-path in $G'- C$, a contradiction. Thus, such an independent set $I$ does not exist, and so by Observation \ref{obs-trans-intersection}, we deduce that $(G,d,k)$ is a \yes-instance of \trans($\alpha$).

Let now $(G,d,k)$ be a \yes-instance of \trans($\alpha$). We want to show that $(G',k)$ is a \yes-instance of \mincut. Let $S\subseteq V$, with $|S|\leq k$, be a $d$-transversal of $G$. We may assume that $S$ is minimal.

We iteratively construct a set $C$ using Algorithm~\ref{AlgoCutconstruction}, and we will prove that $C$ is an $s$-$t$-cut in $G'$ with $|S| = |C|$.  For each vertex in $S$, the algorithm chooses the corresponding vertex in $G'$ belonging to the lowest level such that there is an $s$-$t$-path in $G'- C$ containing this vertex, and then adds it to $C$.  We will find such a vertex for every vertex in $S$, since otherwise $S$ would not be minimal. Hence, it is clear that $|S| = |C|$.

\begin{algorithm}[ht]
\caption{}
\label{AlgoCutconstruction}
\begin{algorithmic}[2]
\Require The graph $G'$ constructed from a co-comparability graph $G$,\\ a minimal $d$-transversal $S$ in $G$.
\Ensure An $s$-$t$-cut $C \subseteq V'$ with $|S| = |C|$.
\State Let $S = \set{u_1,\dots, u_{|S|}}, u_i \prec u_j$ for $i < j$, $i,j \in\intInterval{|S|}$.
\State Let $C = \varnothing$.
\For {$i $ from $1\to |S|$}
	\State Let $u = u_i$.
	\State Let $y_1,\dots, y_{d}\in V'$ be the vertices corresponding to $u$, sorted by increasing level.
	\For{$j$ from $1 \to d$}
		\If {$\exists$ $s$-$t$-path in $G'- C$ containing $y_j$}
			\State $C = C \cup y_j$
			 \State \textbf{break}
		\Else
			\State \textbf{continue}
		\EndIf
	\EndFor
\EndFor
\end{algorithmic}
\end{algorithm}

To prove that $C$, which is constructed by applying Algorithm~\ref{AlgoCutconstruction}, is indeed an $s$-$t$-cut in~$G'$, we assume for a contradiction that there exists an $s$-$t$-path $P_1$ in $G'- C$.  Let $I_1\subseteq V$ be the max-extendable independent set in $G$ of size $\alpha-d+1$ corresponding to $P_1$ ($I_1$ exists by Lemma~\ref{indsetsarepaths}).
Let $I_1^e\subseteq V$ be a maximum extension of $I_1$ in $G$. Since $S$ is a $d$-transversal of $G$, we know that $S\cap I_1 \neq \varnothing$. Let $v \in S\cap I_1$, and if $|S\cap I_1| > 1$, we choose the rightmost vertex according to $\prec$ in $S\cap I_1$ for $v$. Let $y_j\in V'$, for $j\in \intInterval{d}$, be the copy of $v$ in $P_1$. Since $P_1$ is a path in $G'- C$, we know that $y_j \not\in C$. Hence, there is some other copy of $v$ in $G'$, say $y_i$, $i \in \intInterval{d}$,  such that Algorithm~\ref{AlgoCutconstruction} added $y_i$ to $C$. Due to the procedure we use in Algorithm~\ref{AlgoCutconstruction} to choose the vertices in $C$, we have $i < j$.

Since $y_i$ was added to $C$, there exists a path $P_2$ in $G'$ containing $y_i$. Let $I_2$ be the max-extendable independent set of size $\alpha-d+1$ in $G$ corresponding to $P_2$ ($I_2$ exists by Lemma \ref{indsetsarepaths}), and let $I_2^e$ be a maximum extension of $I_2$ in $G$. Let $J_1 \subseteq I_1 $ be the set of the $\alpha-d+1-\pos_{I_1}(v)$ rightmost vertices in $I_1$, i.e.\ those vertices $u$ in $I_1$ such that $\pos(v)< \pos(u)$ (see Figure~\ref{compositionIndset}a)). We know that $J_1 \cap S = \varnothing$ by the choice of $v$. Let $J_1^e \supseteq J_1 $ be the set of the $\alpha-\pos(v)$ rightmost vertices in $I_1^e$, i.e.\ those vertices $u$ in $I_1^e$ such that $v\prec u$ (see Figure~\ref{compositionIndset}a)). Similarly, let $J_2\subseteq I_2$ be the set of the first $\pos_{I_2}(v)-1$ vertices in $I_2$, i.e.\ those vertices $u$ in $I_2$ such that $\pos(u)< \pos(v)$ (see Figure~\ref{compositionIndset}b)), and let $J_2^e \supseteq J_2$ be the set of the first $\pos(v)-1$ vertices in $I_2^e$, i.e.\ those vertices $u$ in $I_2^e$ such that $\pos(u)<\pos(v)$ (see Figure~\ref{compositionIndset}b)). It then follows from Property~\ref{prop-transitive} that $J_2^e \cup \set{v} \cup J_1^e$ is an independent set in $G$ of size $(\alpha-\pos(v))+1+(\pos(v)-1)=\alpha$, i.e.\ a maximum independent set in $G$. Thus, $J_2 \cup J_1$ is a max-extendable independent set in $G$. It follows from Property~\ref{prop-trans-levelpos}, that $\pos_{I_1}(v) < \pos_{I_2}(v)$. We conclude that $|J_2\cup J_1| =(\pos_{I_2}(v)-1)+(\alpha-d+1-\pos_{I_1}(v)) > \alpha-d$. Hence, $J_2 \cup J_1$ either is a max-extendable independent set in $G-S$ of size at least $\alpha-d+1$, or $J_2 \cap S \neq \varnothing$.  In the first case, we directly get a contradiction to our assumption that $S$ is a $d$-transversal in $G$.  So, we may assume that $J_2 \cap S \neq \varnothing$.

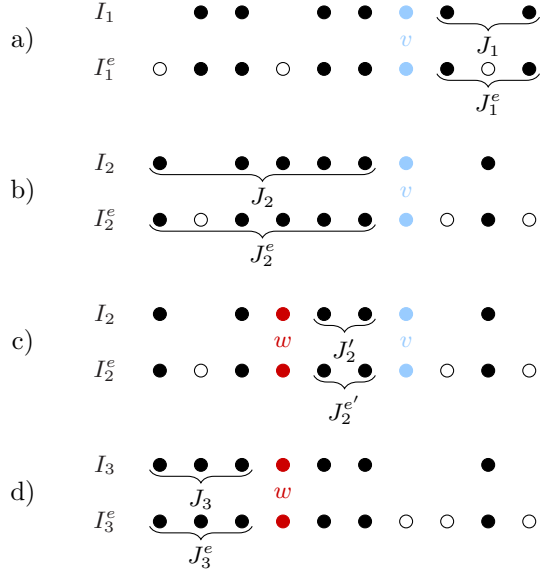
\begin{figure}[ht]
\centering
\begin{tikzpicture}
   
	\tikzstyle{interval} = [thick]
	\tikzstyle{intervald} = [thick, dotted]
	\tikzstyle{intervalv} = [thick, RoyalBlue]
	\tikzstyle{intervalu} = [thick, ForestGreen]
	\tikzstyle{intervalr} = [thick, Red]
	\tikzstyle{interval1d} = [thick, ForestGreen, dotted]
	\tikzstyle{indsetJ} = [thick, Red]
	\def\ha{3}
	\def\hb{2.25}
	\def\hc{1.5}
	
	\begin{scope}[ xscale = 0.9]
	\node[vertex](a) at (-0.4,\ha){};	
	\node[vertex](a) at (0.2,\ha){};	
	\node[vertex](a) at (1.4,\ha){};	
	\node[vertex](a) at (2,\ha){};	
	\node[bvertex](a) at (2.6,\ha){};	
	\node[vertex](a) at (3.2,\ha){};	
	\node[vertex](a) at (4.4,\ha){};		
	
	\node[evertex](a) at (-1,\hb){};
	\node[vertex](a) at (-0.4,\hb){};	
	\node[vertex](a) at (0.2,\hb){};	
	\node[evertex](a) at (0.8,\hb){};	
	\node[vertex](a) at (1.4,\hb){};	
	\node[vertex](a) at (2,\hb){};	
	\node[bvertex](a) at (2.6,\hb){};	
	\node[vertex](a) at (3.2,\hb){};	
	\node[evertex](a) at (3.8,\hb){};	
	\node[vertex](a) at (4.4,\hb){};

	\node[font = \small, lightblue] at (2.6,2.625){$v$};
	
	\node[font = \small, Black] at (-1.8,\ha){$I_1$};
	\node[font = \small, Black] at (-1.8,\hb){$I^e_1$};

	\draw [decorate,decoration={brace,mirror,amplitude=5pt},xshift=0pt,yshift=11pt] (3.05,2.55) -- (4.55,2.55) node [black,midway,yshift=-4pt, below] {\small $J_1$};
		\draw [decorate,decoration={brace,mirror,amplitude=5pt},xshift=0pt,yshift=11pt] (3.05,1.8) -- (4.55,1.8) node [black,midway,yshift=-4pt, below] {\small $J^e_1$};

	\node[] at (-3,2.625){a)};
	\end{scope}
	
	\begin{scope}[shift = {(0,-2)}, xscale = 0.9]
		\node[vertex](a) at (-1,\ha){};
	\node[vertex](a) at (0.2,\ha){};	
	\node[vertex](a) at (0.8,\ha){};	
	\node[vertex](a) at (1.4,\ha){};	
	\node[vertex](a) at (2,\ha){};	
	\node[bvertex](a) at (2.6,\ha){};	
	\node[vertex](a) at (3.8,\ha){};

		\node[vertex](a) at (-1,\hb){};
	\node[evertex](a) at (-0.4,\hb){};	
	\node[vertex](a) at (0.2,\hb){};	
	\node[vertex](a) at (0.8,\hb){};	
	\node[vertex](a) at (1.4,\hb){};	
	\node[vertex](a) at (2,\hb){};	
	\node[bvertex](a) at (2.6,\hb){};	
	\node[evertex](a) at (3.2,\hb){};	
	\node[vertex](a) at (3.8,\hb){};	
	\node[evertex](a) at (4.4,\hb){};	
	
	\node[font = \small, lightblue] at (2.6,2.625){$v$};
		\node[font = \small, Black] at (-1.8,\ha){$I_2$};
		\node[font = \small, Black] at (-1.8,\hb){$I_2^e$};
	
		\draw [decorate,decoration={brace, mirror,amplitude=5pt},xshift=0pt,yshift=11pt] (-1.15,1.8) -- (2.15,1.8) node [black,midway,yshift=-4pt, below] {\small $J_2^e$};
	\draw [decorate,decoration={brace, mirror,amplitude=5pt},xshift=0pt,yshift=11pt] (-1.15,2.55) -- (2.15,2.55) node [black,midway,yshift=-4pt, below] {\small $J_2$};
	\node[] at (-3,2.625){b)};
	\end{scope}
	
		\begin{scope}[shift = {(0,-4)}, xscale = 0.9]
		\node[vertex](a) at (-1,\ha){};
	\node[vertex](a) at (0.2,\ha){};	
	\node[rvertex](a) at (0.8,\ha){};	
	\node[vertex](a) at (1.4,\ha){};	
	\node[vertex](a) at (2,\ha){};	
	\node[bvertex](a) at (2.6,\ha){};	
	\node[vertex](a) at (3.8,\ha){};

		\node[vertex](a) at (-1,\hb){};
	\node[evertex](a) at (-0.4,\hb){};	
	\node[vertex](a) at (0.2,\hb){};	
	\node[rvertex](a) at (0.8,\hb){};	
	\node[vertex](a) at (1.4,\hb){};	
	\node[vertex](a) at (2,\hb){};	
	\node[bvertex](a) at (2.6,\hb){};	
	\node[evertex](a) at (3.2,\hb){};	
	\node[vertex](a) at (3.8,\hb){};	
	\node[evertex](a) at (4.4,\hb){};	
	
	\node[font = \small, lightblue] at (2.6,2.625){$v$};
	\node[font = \small, nicered] at (0.8,2.625){$w$};
		\node[font = \small, Black] at (-1.8,\ha){$I_2$};
		\node[font = \small, Black] at (-1.8,\hb){$I_2^e$};
	
		\draw [decorate,decoration={brace, mirror,amplitude=5pt},xshift=0pt,yshift=11pt] (1.25,1.8) -- (2.15,1.8) node [black,midway,yshift=-4pt, below] {\small $J_2^{e'}$};
	\draw [decorate,decoration={brace, mirror,amplitude=5pt},xshift=0pt,yshift=11pt] (1.25,2.55) -- (2.15,2.55) node [black,midway,yshift=-4pt, below] {\small $J'_2$};
	\node[] at (-3,2.625){c)};
	\end{scope}
	
	\begin{scope}[shift = {(0,-6)}, xscale = 0.9]

		\node[vertex](a) at (-1,\ha){};
	\node[vertex](a) at (-0.4,\ha){};	
	\node[vertex](a) at (0.2,\ha){};	
	\node[rvertex](a) at (0.8,\ha){};	
	\node[vertex](a) at (1.4,\ha){};	
	\node[vertex](a) at (2,\ha){};	
	\node[vertex](a) at (3.8,\ha){};

	\node[vertex](a) at (-1,\hb){};
	\node[vertex](a) at (-0.4,\hb){};	
	\node[vertex](a) at (0.2,\hb){};	
	\node[rvertex](a) at (0.8,\hb){};	
	\node[vertex](a) at (1.4,\hb){};	
	\node[vertex](a) at (2,\hb){};	
	\node[evertex](a) at (2.6,\hb){};	
	\node[evertex](a) at (3.2,\hb){};	
	\node[vertex](a) at (3.8,\hb){};	
	\node[evertex](a) at (4.4,\hb){};

	\node[font = \small, nicered] at (0.8,2.625){$w$};

	\node[font = \small, Black] at (-1.8,\ha){$I_3$};
	\node[font = \small, Black] at (-1.8,\hb){$I_3^e$};
	
	\draw [decorate,decoration={brace, mirror,amplitude=5pt},xshift=0pt,yshift=11pt] (-1.15,2.55) -- (0.35,2.55) node [black,midway,yshift=-4pt, below] {\small $J_3$};
		\draw [decorate,decoration={brace, mirror,amplitude=5pt},xshift=0pt,yshift=11pt] (-1.15,1.8) -- (0.35,1.8) node [black,midway,yshift=-4pt, below] {\small $J_3^e$};

	\node[] at (-3,2.625){d)};

	\end{scope}

\end{tikzpicture}
%\begin{tikzpicture}
%
%	\tikzstyle{rahmenb}=[rounded corners = 5pt,draw, dashed, minimum width = 310pt, minimum height = 30pt]
%	\tikzstyle{rahmeng}=[rounded corners = 5pt,draw=ForestGreen, dashed, minimum width = 142.5pt, minimum height = 30pt]
%	\tikzstyle{vertexsetb}=[rounded corners = 4pt,draw, minimum width = 10pt, minimum height = 25pt]
%	\tikzstyle{vertexsetg}=[rounded corners = 4pt,draw=ForestGreen, minimum width = 10pt, minimum height = 25pt,]
%	\tikzstyle{interval} = [thick]
%	\tikzstyle{intervalv} = [thick, RoyalBlue]
%	\tikzstyle{intervalu} = [thick, ForestGreen]
%	
%	\node[vertex](s) at (0,0){};
%	\node[vertex](t) at (10,5){};
%	\node[rvertex](u1) at (3,2){};
%	\node[rvertex](u2) at (3,1){};
%	\node[bvertex](vi) at (6,3){};
%	\node[bvertex](vj) at (6,4){};
%	
%	\begin{scope}[on behind layer]
%		\draw[edge] plot [smooth] coordinates {(s) (2.5,2.5) (vj)  (t)};
%		\draw[edge] plot [smooth] coordinates {(s)(u1) (vi)  (t)};
%		\draw[edge] plot [smooth] coordinates {(s)(u2) (7.5,2.5)  (t)};
%	\end{scope}
%	
%	
%	\node[font = \small, Black, below] at (0,0){$s$};
%\node[font = \small, Black, above] at (10,5){$t$};	
%\node[font = \small, Black,] at (8,4.8){$P$};
%\node[font = \small, Black] at (8,3.6){$P_1$};	
%\node[font = \small, Black,] at (8,2.5){$P_2$};
%
%\node[font = \small, Black, below] at (6,3){$x_i$};
%\node[font = \small, Black, above] at (6,4){$x_j$};	
%\node[font = \small, Black, below] at (3,1){$z_g$};
%\node[font = \small, Black, above] at (3,2){$z_h$};	
%
%\node[] at (5,-0.5){c)};
%	
%\end{tikzpicture}
\caption{\label{compositionIndset}The different independent sets, considered in the proof. Note that the empty vertices are only part of the corresponding extended version of the independent set (e.g. in $I_1^e$ but not in $I_1$), while the other vertices are contained in both sets.}
\end{figure}

Let $w\in J_2 \cap S$, and if $|J_2 \cap S|>1$, we take the rightmost vertex $w$ in $J_2\cap S$ with respect to~$\prec$ such that $\pos(w)<\pos(v)$. 
Let $y_h$, $h\in \intInterval{d}$, be the vertex in $P_2$ corresponding to $w$. Then, $y_h \not\in C$, since otherwise Algorithm~\ref{AlgoCutconstruction} would not have added $y_i$ to $C$. Thus, as before, there exists some vertex $y_g$, $g\in \intInterval{d}$, with $g<h$, such that $y_g$ corresponds to $w$ and $y_g\in C$. 
Therefore, there exists a path $P_3$ in $G'$ containing $y_g$. Let $I_3$ be the max-extendable independent set of size $\alpha-d+1$ in $G$ corresponding to $P_3$, and let $I_3^e$ be its extension (see Figure~\ref{compositionIndset}(d)). We define $J_3 = \set{u \in I_3|\pos_{I_3}(u) < \pos_{I_3}(w)}$, as well as $J_3^e=\set{u \in I_3|\pos(u) < \pos(w)}\supseteq J_3$. 
Furthermore, we consider the set $J_2' = \set{u \in I_2|\pos_{I_2}(w) < \pos_{I_2}(u) < \pos_{I_2}(v)}$ as well as {$J_2^{e'} = \set{u \in I_2|\pos(w) < \pos(u) < \pos(v)}\supseteq J_2'$ (see Figure~\ref{compositionIndset}(d)). }
Thus, $J_3\cup J_2'\cup J_1$ is a subset of {$J_3^e\cup \{w\}\cup J_2^{e'}\cup \{v\}\cup J_1^e$}, which by Property~\ref{prop-transitive} is an independent set in $G$ of size $(\pos(w)-1)+1+(\pos(v)-\pos(w)-1)+1+(\alpha-\pos(v))=\alpha$, i.e.\ a maximum independent set. 
Thus, $J_3\cup J_2'\cup J_1$ is a max-extendable independent set in $G$. It follows from Property~\ref{prop-trans-levelpos}, that $\pos_{I_2}(w) < \pos_{I_3}(w)$. 
Since in addition $\pos_{I_1}(v) < \pos_{I_2}(v)$ (see above), we conclude that $|J_3 \cup J_2' \cup J_1|=(\pos_{I_3}(w)-1)+(\pos_{I_2}(v)-\pos_{I_2}(w)-1)+(\alpha-d+1-\pos_{I_1}(v))>\alpha-d$. 
Hence, we obtain again that either $J_3 \cup J_2' \cup J_1$ is a max-extendable independent set in $G-S$ of size at least $\alpha-d+1$, or that $J_3\cap S\neq \varnothing$.
As before, the first case gives us a contradiction to our assumption that $S$ is a $d$-transversal in $G$. So, we may assume that $J_3\cap S\neq \varnothing$.

By repeatedly using these arguments, we can always find a new vertex in $S$. But since $S$ is finite, this case can not always occur. Hence, we will necessarily get a contradiction and thus, there is no $s$-$t$-path $P$ in $G'- C$. So we conclude that $C$ is an $s$-$t$-cut in $G'$.

Let us now consider the complexity of our algorithm. From \cite{tarjanmaxflow}, we know that for a graph with $n$ vertices, we can solve \mincut \ in $\mathcal{O}(n^3)$. Since the graph $G'$ has $\mathcal{O}(dn)$ vertices, computing a \mincut \ in $G'$ can be done in time $\mathcal{O}(d^3n^3)$. We still need to consider the time we need to construct $G'$. Using Lemma~\ref{lem-trans-partition-poly}, we know that we can find the partition of $\mathcal{I}$ into sets $L_{p}$, for $p\in \intInterval{d}$, in time $\mathcal{O}(|V|^2)$. For every pair of vertices in $G'$, we can check in $\mathcal{O}(|V|)$ if we introduce an arc between them in $G'$. Hence, $G'$ can be constructed in $\mathcal{O}(d^2|V|^3)$. We conclude that \trans($\alpha$) can be solved in time $\mathcal{O}(d^3|V|^3)$.
\end{proof}

%============================================================================================
%============================================================================================
%============================================================================================
%============================================================================================
%============================================================================================
%============================================================================================

\section{More Properties of Independent Sets in Co-Comparability Graphs}
\label{sec-structure-nonmax}
%When considering \del$(\alpha)$, we can make an observation similar to Observation~\ref{obs-trans-intersection}.
To solve \del$(\alpha)$,  we will use an approach similar to the one in Section~\ref{sec:trans}. 
This requires an extension of the structural results from Section~\ref{sec:indepsettrans} for independent sets that are not necessarily maximum. We sometimes omit proofs since they are very similar to the ones we presented in the previous sections.%To do this we examine the structure of non-maximum independent sets and we start with an observation similar to Observation~\ref{obs-trans-intersection}.

Recall that Lemma~\ref{lemma-fixedposition} allowed us to partition the vertices in a maximum independent set by using their position. The same property holds for all vertices with respect to the largest independent set in which they are contained.
\begin{lemma}
\label{lem-fixedposition-nonmax}
Let $G = (V,E)$ be a co-comparability graph with a vertex ordering $\prec$. Let $v \in V$ and let $I_1,I_2$ be two independent sets of $G$ such that both contain $v$ and they are maximum among the independent sets of $G$ containing $v$. Let $\beta = |I_1| = |I_2|$.
Let $I_1 = \set{u_1, \dots, u_{i-1}, v = u_i, u_{i+1}, \dots, u_\beta}$, with $u_1 \prec \dots \prec u_\beta$, and $I_2= \set{v_1, \dots, v_{j-1}, v = v_j, v_{j+1}, \dots, v_\beta}$, with $v_1 \prec \dots \prec v_\beta$. Then $i = j$.
\end{lemma}
\begin{proof}
Suppose that $i < j$, then similar to the proof of Lemma~\ref{lemma-fixedposition}, we can find an independent set $\set{v_1, \dots, v_{j-1}, v, u_{i+1}, \dots, u_\beta}$ which has size $\beta +1$ and contains $v$, a contradiction.
\end{proof}

\begin{definition}
\label{def-partition}
Let $G=(V,E)$ be a co-comparability graph with $\alpha = \alpha(G)$.  Let $L_1,\dots, L_\alpha$ be as in Lemma~\ref{lem-trans-partition-poly}. 
Let $\I_{\beta}$ be the set of vertices that occur in an independent set of size $\beta$, but not of size $\beta+1$, $\beta \in \intInterval{\alpha}$. 
We define $L_{p,\beta}, p \in\intInterval{\beta}$, as the set of vertices $v \in \I_{\beta}$ such that for an independent set $I$ of size $\beta$ containing $v \in V$ we have that $\pos_I(v) = p$.
%that are complete to $\bigcup_{i\in \{p,\dots, p+\alpha-\beta\}} L_{i} $ but not to any other $L_{i}$. 
From now on, we refer to the sets $L_p$ from Lemma~\ref{lem-trans-partition-poly} as $L_{p,\alpha}$. We say that $\pos(v) = p$.
\end{definition}

Note that we defined the position of a vertex in two different ways. We will consider in the following both the relative position $\pos_I$ of a vertex, which depends on the independent set $I$ and the absolute position $\pos$ of a vertex, which is independent of any specific independent set. Figure \ref{f-blocker-example} gives an example of a co-comparability graph. We can see the assignment of the vertices in $\I_{\alpha}$ and $\I_{\alpha-1}$ to the sets $L_{p,\beta}$. The set $L_{3,\alpha-1}$ is empty.

\begin{observation}
\label{obs-blocker-size-extensions}
Let $G=(V,E)$ be a co-comparability graph and let $v \in L_{p,\beta}, \beta \in \intInterval{\alpha}, p \in\intInterval{\beta}$. Consider a maximum left extension $I_\ell^v$ of $v$ and a maximum right extension $I_r^v$ of $v$. Then we have that $|I_\ell^v| = p-1$ and $|I_r^v| = \beta - |I_\ell^v| - 1 = \beta - (p-1) - 1 = \beta -p$.
\end{observation}

From Property~\ref{prop-trans-partition-clique} we know that each of the sets $L_{p,\alpha}$ for $p \in \intInterval{\alpha}$ is a clique. We will generalise this to the sets $L_{p,\beta}$, for $\beta \in \intInterval{\alpha}$.
\begin{property}
\label{prop-blocker-partitionclique}
Let $G=(V,E)$ be a co-comparability graph with $\alpha = \alpha(G)$.  Let $L_{p,\beta}$, with $\beta \in \intInterval{\alpha}, p \in \intInterval{\beta}$ be as in Definition~\ref{def-partition}.
Let $v \in L_{p,\beta}$. Then, $v$ is adjacent to all vertices in \[\bigcup_{ j \in \set{\beta,\dots, \alpha}, i \in \set{0, \dots,  j-\beta}} L_{p+i,j}.\]
\end{property}
\begin{proof}
Let $v \in L_{p,\beta}$. 
Assume there exists $u \in L_{q,\gamma}$, for $\gamma \in \set{\beta, \dots, \alpha}$, $q \in \set{p,\dots, p+\gamma-\beta}$ such that $uv \notin E$. 
%By definition we know that $u \notin \bigcup_{i \in [\alpha]} L_{i,\alpha}$.
We assume now that $v \prec u$. The case $u \prec v$ can be handled in a similar way.
Let $I_\ell^v$ be a maximum left extension of~$v$ and let $I_r^u$ be a maximum right extension of $u$.
From Observation~\ref{obs-extensions} we obtain that $I_\ell^v$ is a left extension of $u$ and $I_r^u$ is a right extension of $v$.
This gives us $I_\ell^v \prec v \prec u \prec I_r^u$ and thus, by Property~\ref{prop-transitive}, $I = I_\ell^v \cup \set{u,v} \cup I_r^u$ is an independent set.
Since $u \in L_{q,\gamma}$ we have, by Observation~\ref{obs-blocker-size-extensions}, that $|I_r^u| = \gamma-q \geq \gamma -(p + \gamma -\beta) = \beta-p$.
Furthermore, Observation~\ref{obs-blocker-size-extensions} tells us that $|I_\ell^v| = p-1$, since $v \in L_{p,\beta}$.
Thus, $|I| \geq p-1 + 2 + \beta-p = \beta+1$, a contradiction to $v \in I_\beta$.
\end{proof}

We can see in Figure~\ref{f-blocker-example} that vertex $v \in L_{1,\alpha-1}$ is complete to $L_{1,\alpha}$ and $L_{2,\alpha}$. 

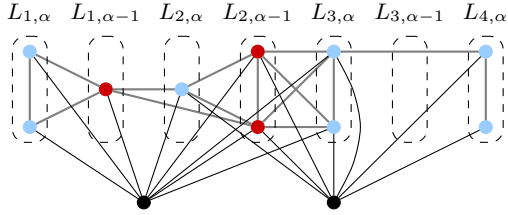
\begin{figure}[ht]
\begin{center}
\centering
  \begin{tikzpicture}

%\tikzstyle{rahmen1}=[rounded corners = 5pt,draw, dashed, minimum width = 13pt, minimum height = 13pt]
\tikzstyle{rahmen2}=[rounded corners = 5pt,draw, dashed, minimum width = 13pt, minimum height = 40pt]
\tikzstyle{nodename}=[font=\footnotesize]
          
    \node[bvertex] (a1) at (0,4.5){};
    \node[bvertex] (a2) at (0,3.5){};
    \node[rvertex] (a3) at (1,4){};
    \node[bvertex] (a4) at (2,4){};
    \node[rvertex] (a5) at (3,4.5){};
    \node[rvertex] (a6) at (3,3.5){};
    \node[bvertex] (a7) at (4,4.5){};
    \node[bvertex] (a8) at (4,3.5){};
    
    \node[bvertex] (a9) at (6,4.5){};
    \node[bvertex] (a10) at (6,3.5){};
    
    \node[vertex] (a11) at (1.5,2.5){};
    \node[vertex] (a12) at (4,2.5){};
    
    \draw[edge] (a1)--(a2);
    \draw[edge] (a2)--(a3);
    \draw[edge] (a1)--(a3);
    \draw[edge] (a3)--(a4);
    \draw[edge] (a3)--(a6);
    \draw[edge] (a4)--(a5);
    \draw[edge] (a4)--(a6);
    \draw[edge] (a5)--(a6);
    \draw[edge] (a5)--(a7);
    \draw[edge] (a5)--(a8);
    \draw[edge] (a6)--(a7);
    \draw[edge] (a6)--(a8);
    \draw[edge] (a7)--(a8);
    \draw[edge] (a7)--(a9);
    \draw[edge] (a9)--(a10);

    \foreach \i in {1,...,8}{
		  \draw[thinedge] (a11)--(a\i);
		}
		\foreach \i in {4,5,6,8,9,10}{
		  \draw[thinedge] (a12)--(a\i);
		}
		\begin{scope}[on background layer]
    		\draw[thinedge] plot [smooth] coordinates {(a12) (4.25,  3) (4.35, 3.5) (4.25, 4) (a7)};
    		\end{scope}

	\node[rahmen2]at (0,4){};	
	\node[rahmen2]at (1,4){};	
	\node[rahmen2]at (2,4){};
	\node[rahmen2]at (3,4){};	
	\node[rahmen2]at (4,4){};
	\node[rahmen2]at (5,4){};
	\node[rahmen2]at (6,4){};	

\node[nodename] at (0,5){$L_{1,\alpha}$};
\node[nodename] at (1,5){$L_{1,\alpha-1}$};
\node[nodename] at (2,5){$L_{2,\alpha}$};
\node[nodename] at (3,5){$L_{2,\alpha-1}$};
\node[nodename] at (4,5){$L_{3,\alpha}$};
\node[nodename] at (5,5){$L_{3,\alpha-1}$};
\node[nodename] at (6,5){$L_{4,\alpha}$};

\end{tikzpicture}

%\end{document}
\caption{\label{f-blocker-example} The figure shows a co-comparability graph. The blue vertices are those in $\I_\alpha$, the red ones those in $\I_{\alpha-1}$. The black vertices are contained in neither $\I_\alpha$ nor in $\I_{\alpha-1}$. Note that in this example $\alpha = \alpha(G) = 4$.}
\end{center}
\end{figure}

\begin{lemma}
\label{lemma-blocker-starttoordering}
Let $G = (V,E)$ be a co-comparability graph with vertex ordering $\prec$, $\alpha = \alpha(G)$.
Let $u,v \in V$, $uv \notin E$ and $\pos(u) < \pos(v)$.
Then $u \prec v$.
\end{lemma}

\begin{proof}
Let $I_\ell^v$ be a maximum left extension of $v$ and $I_r^u$ be a maximum right extension of~$u$.
Suppose for a contradiction that $v \prec u$.
By definition of a left extension (resp. right extension) we get that $I_\ell^v \prec v$ (resp. $u \prec I_r^u $).
Since $v \prec u$ it follows from Property~\ref{prop-transitive} that $I_\ell^v \prec v \prec u \prec I_r^u$. 
Further, since $I_\ell^v \cup v$ and $u \cup I_r^u $ are both independent sets and $uv \notin E$ the set $I = I_\ell^v \cup v \cup u \cup I_r^u$ is an independent set.
Let $\beta_u$ be the size of a maximum independent set in $G$ containing $u$.
From Observation~\ref{obs-blocker-size-extensions} we get that
\[|I| = |I_\ell^v| + 2 + |I_r^u| = \pos(v)-1 + 2 + \beta_u - \pos(u) = \pos(v)- \pos(u) + 1 + \beta_u \geq \beta_u + 1\]
a contradiction to $\beta_u$ being the size of a maximum independent set containing $u$.
Hence, $u \prec v$.
%Consider $u' \in L_{\pos(u),\alpha}$ such that $u'v \notin E$. Note that if $u \in I_\alpha$, we might have $u = u'$. 
%Assume that $v \prec u'$. Let $w \in L_{\pos(v), \alpha}$ such that $\set{u',w}$ is extendable to a maximum independent set of $G$. By definition of the sets $L_{p,\alpha}$, $p \in \intInterval{\alpha}$, we know that $w$ exists and $u' \prec w$.
%So we have that $v \prec u' \prec w$ while $vu', u'w \notin E$, but by Property~\ref{prop-blocker-partitionclique}, we have $vw \in E$, a contradiction to Property~\ref{prop-transitive}. Thus, we conclude that $u' \prec v$. Hence, if $u  = u'$ we have that $u\prec v$ and the lemma follows.
%So we may assume now that $u' \neq u$ and $v \prec u$. Then, again, we get a contradiction to Property~\ref{prop-transitive}, since $vu, u'v \notin E$ but $u'u \in E$.
%Thus, $u \prec v$.
\end{proof}

\begin{lemma}
\label{lemma-blocker-orderingtostart}
Let $G = (V,E)$ be a co-comparability graph with vertex ordering $\prec$, $\alpha = \alpha(G)$.
Let $u,v \in V$, $uv \notin E$ and $u \prec v$.
Then $\pos(u) < \pos(v)$.
\end{lemma}
\begin{proof}
Suppose for a contradiction that $\pos(u) \geq \pos(v)$. 
Consider first the case where $\pos(u) = \pos(v) = p$, for some $p \in \intInterval{\alpha}$. There are $\beta, \beta' \in \{p,\dots, \alpha\}$ such that $u \in L_{p,\beta}$ and $v \in L_{p,\beta'}$.Without loss of generality we may assume that $\beta < \beta'$. Then from Property~\ref{prop-blocker-partitionclique} it follows that $uv \in E$, a contradiction. Thus, $\pos(u) > \pos(v)$. Further, Lemma~\ref{lemma-blocker-starttoordering} tells us that $v \prec u$, a contradiction.
\end{proof}

\begin{lemma} \label{lem-blocker-computepartition}
The partition of $V$ into sets $L_{p,\beta}$ can be found in $\mathcal{O}(|V|^2)$.
\end{lemma}
\begin{proof}
To see that we can find this partition in polynomial time, recall the proof of Lemma~\ref{lem-trans-partition-poly}.
We defined $\leftext(v)$ respectively $\rightext(v)$ as the size of a left respectively right extension of $v$ of maximum size for any vertex $v\in V$. We showed that we can calculate both functions in $\mathcal{O}(|V|^2)$.
From Property~\ref{prop-transitive}, we know that for a vertex $v \in V$, we can combine every left extension $I_\ell$ of $v$ and every right extension $I_r$ of $v$ together with $v$ to an independent set $I_\ell \cup \set{v} \cup I_r$.
Consider a maximum independent set containing $v$. This consists of $v$ together with a maximum left extension and a maximum right extension of $v$.
Thus, the values computed in $\leftext(v)$ and $\rightext(v)$ lead to a maximum independent set containing~$v$ and $\pos(v) = \leftext(v) + 1$.
\end{proof}

%============================================================================================

\section{Deletion Blocker in Co-Comparability Graphs}
\label{sec-blocker}
In this section, we will show that \del$(\alpha)$ can be solved in polynomial time in co-comparability graphs. 
Let $G = (V,E)$ be a co-comparability graph with vertex ordering~$\prec$, let $\alpha = \alpha(G)$ and let $d>0$ be an integer.  
We will construct a directed graph $G'= (V', A')$ such that $(G',k)$ is a \yes-instance of \mincut \ if and only if $(G,d,k)$ is a \yes-instance of \del($\alpha$), for some integer $k>0$. 
The construction is similar to the one in Section~\ref{sec:trans}. We will show the equivalence between the two instances in Theorem~\ref{theo-blockeralpha}.
We can make an observation similar to Observation~\ref{obs-trans-intersection}.

\begin{observation}\label{obs-blocker-intersection}
Let $G = (V,E)$ be a co-comparability graph with a vertex ordering $\prec$. Let $S \subseteq V$.
$S$ is a $d$-deletion blocker of $G$ if and only if it contains at least one vertex from every independent set of size $\alpha-d+1$ of $G$.
\end{observation}

By Observation~\ref{obs-blocker-intersection} it suffices to consider independent sets of size $\alpha-d+1$.
Thus, in the following, we construct a directed graph $G'$ with source $s$ and sink $t$, such that every $s$-$t$-path in $G'$ corresponds to an independent set of size $\alpha-d+1$ in $G$. 
The one-to-one correspondence will be proven in Lemma~\ref{lem-blocker-indsetsarepaths}. 
An $s$-$t$-cut will then correspond to a $d$-deletion blocker in $G$ (see Theorem~\ref{theo-blockeralpha}).

Let us now describe the construction of $G' = (V', A')$.
Recall that $\I_\beta = \bigcup_{p \in \intInterval{\alpha}} L_{p,\beta}$ is the set of vertices which are contained in an independent set of size $\beta$ but not of size~$\beta +1$, for $\beta \in \intInterval{\alpha}$. 
The vertex set $V'$ of $G'$ consists of $d$ copies $U_{1,\beta},\dots, U_{d,\beta}$ of $\I_\beta$, for all $\beta \in \set{\alpha-d+1, \dots, \alpha}$, and two additional vertices $s,t$. That is 
\[V' = \bigcup_{ \beta \in \set{\alpha-d+1, \dots, \alpha},\ell \in \intInterval{d}} U_{\ell, \beta} \cup \set{s,t}.\] 
We denote by $L_{p,\ell, \beta}$ the set $L_{p,\beta}$ in $U_{\ell,\beta}$, for $p \in \intInterval{\alpha}$, $\ell\in \intInterval{d}$ and $\beta \in \set{\alpha-d+1, \dots, \alpha}$.
We say that $\ell$ is the \emph{level} of the vertices in $U_{\ell,\beta}$, denoted by $\level(x) = \ell$ for $x \in U_{\ell,\beta}$. 
For simplicity, we adopt the notion of positions for all vertices in $V'\setminus \set{s,t}$, that is, for every vertex $x \in V'\setminus\set{s,t}$ that corresponds to some vertex $v \in L_p$, for $ p \in \intInterval{\alpha}$, we will also use $\pos(x)$ in order to actually refer to the position of $v$, the vertex that $x$ corresponds to in $G$.
Furthermore, we set $\pos(s) = 0$ and $\level(s)= 1$ and for $t$ we have $\pos(t) = \alpha+1$ and $\level(t) = d$.

Let $x,y \in V'\setminus\set{s,t}$, where $x \in L_{p,\ell,\beta}$ and $y \in L_{p',\ell',\beta'}$, with $p, p' \in \intInterval{\alpha}$, $\ell,\ell' \in \intInterval{d}$ and $\beta, \beta' \in \set{\alpha-d+1, \dots, \alpha}$.
Let $u,v \in V$, where $x$ corresponds to $u$ and $y$ corresponds to $v$. We add an arc $(x,y)$ if $uv \notin E$ and $p' = p+g+1$, $\ell' = \ell+g$ for some $g \geq 0$.
Similar to the construction in Section~\ref{sec:trans}, we add an arc $(s,y)$ if $p' = g+1$, $\ell' = g+1$, for some $g \geq 0$ and we add an arc $(x,t)$ if $\pos(t) = p+g+1$, $\level(t) = \ell + g$, for some $g \geq 0$.
We can see that Observation~\ref{obs-trans-construction} holds for this construction as well.

Figures~\ref{f-blocker-galpha} and \ref{f-blocker-gprime-complete} give examples for the graph $G'$ constructed from the graph in Figure~\ref{f-blocker-example} for $d=1$ and $d=2$. 
%We can see in the figures that for an arc $(x,y)$ with $x,y \in V'$ we have $\pos(y) > \pos(x)$ and $\level(y) \geq \level(x)$.
%Further we have that $\pos(y) - \pos(x) -1= \level(y) - \level(x)$. 
Note that some sets $L_{p,\ell,\beta}$ in Figure~\ref{f-blocker-gprime-complete} are not contained in any $s$-$t$-path and thus could be removed from the graph. More such sets appear when $d > 2$. We kept them to make the definitions and proofs easier.

\begin{figure}[ht]
\begin{center}
\centering
  \begin{tikzpicture} 

\tikzstyle{rahmen}=[rounded corners = 5pt,draw, dashed, minimum width = 13pt, minimum height = 40pt]
\tikzstyle{nodename}=[font=\small]

	\node[vertex] (s) at (-2,4){};
          
    \node[vertex] (a1) at (0,4.5){};
    \node[vertex] (a2) at (0,3.5){};

    \node[vertex] (a3) at (2,4){};

    \node[vertex] (a4) at (4,4.5){};
    \node[vertex] (a5) at (4,3.5){};
    
    \node[vertex] (a6) at (6,4.5){};
    \node[vertex] (a7) at (6,3.5){};
    
    \node[vertex] (t) at (8,4){};

    \draw[arc] (a2)--(a3);
    \draw[arc] (a1)--(a3);
    \draw[arc] (a3)--(a4);
    \draw[arc] (a3)--(a5);
    \draw[arc] (a4)--(a7);
    \draw[arc] (a5)--(a6);
    \draw[arc] (a5)--(a7);
    
     \draw[arc] (s)--(a1);
    \draw[arc] (s)--(a2);
    \draw[arc] (a6)--(t);
    \draw[arc] (a7)--(t);

\node[rahmen] (r) at (0,4){};
\node[rahmen] (r) at (2,4){};
\node[rahmen] (r) at (4,4){};
\node[rahmen] (r) at (6,4){};

\node[] (r) at (0,5){$L_{1,1, \alpha}$};
\node[] (r) at (2,5){$L_{1,1, \alpha}$};
\node[] (r) at (4,5){$L_{1,1, \alpha}$};
\node[] (r) at (6,5){$L_{1,1, \alpha}$};

\node[nodename, , below] at (-2,4){$s$};
\node[nodename, , above] at (8,4){$t$};

\end{tikzpicture}
\caption{\label{f-blocker-galpha}The graph $G'$ corresponding to the graph $G$ from Figure \ref{f-blocker-example} for $d=1$.}
\end{center}
\end{figure}
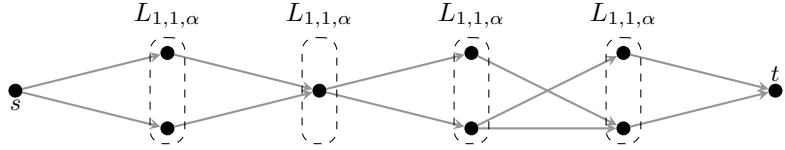

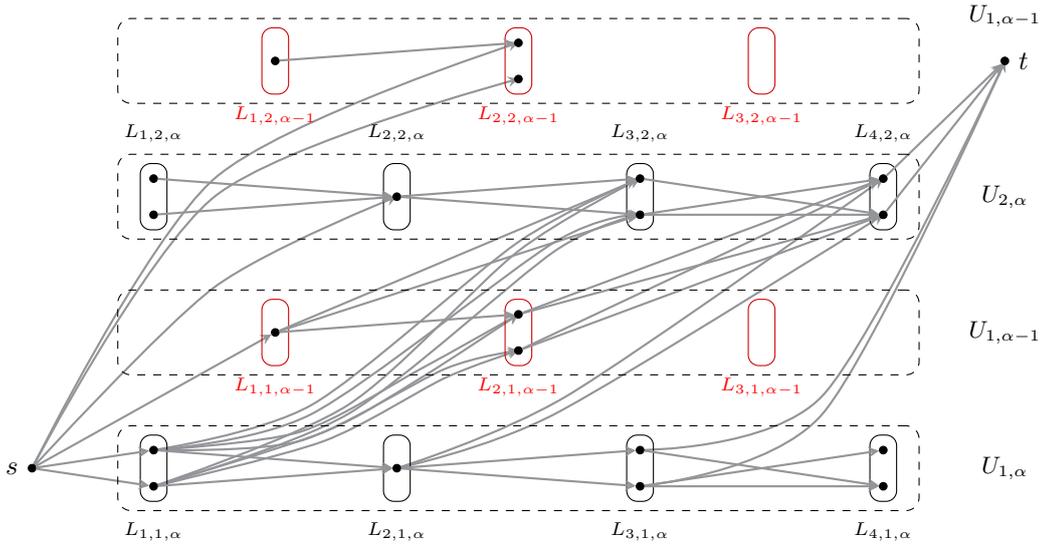
\begin{figure}[ht]
\begin{center}
\centering
  \begin{tikzpicture} 
  \begin{scope}[yscale = 1.2, xscale = 1.6]
  
    %overwrite to make smaller
\tikzstyle{vertex}=[thin,circle,inner sep=0.cm, minimum size=1mm, fill=black, draw=black]
\tikzstyle{rvertex}=[thin,circle,inner sep=0.cm, minimum size=1mm, fill=nicered,draw=nicered]

    %\tikzstyle{gvertex}=[thin,circle,inner sep=0.cm, minimum size=1mm, fill=ForestGreen, draw=ForestGreen]
    %\tikzstyle{bvertex}=[thin,circle,inner sep=0.cm, minimum size=1.7mm, fill=RoyalBlue, draw=RoyalBlue]
     %\tikzstyle{blackvertex}=[thin,circle,inner sep=0.cm, minimum size=1mm, fill=Black, draw=Black]
%\tikzstyle{hedge}=[thick, draw = Gray, >=stealth, shorten > = 0pt, ->]
%\tikzstyle{bedge}=[thick, draw = RoyalBlue, ->, >=stealth]
%\tikzstyle{gedge}=[thick, draw = ForestGreen, ->, >=stealth]

\tikzstyle{rahmen}=[rounded corners = 5pt,draw, dashed, minimum width = 300pt, minimum height = 32pt]
\tikzstyle{rahmenred}=[rounded corners = 5pt,draw=nicered, dashed, minimum width = 208pt, minimum height = 32pt]
\tikzstyle{vertexsetb}=[rounded corners = 4pt,draw, minimum width = 10pt, minimum height = 25pt]
\tikzstyle{vertexsetred}=[rounded corners = 4pt,draw=nicered, minimum width = 10pt, minimum height = 25pt,]
%\tikzstyle{nodename}=[font=\small]

\node[vertexsetb](L11a) at (2,1){};
\node[vertexsetb](L21a) at (4,1){};
\node[vertexsetb](L31a) at (6,1){};
\node[vertexsetb](L41a) at (8,1){};

\node[vertex](v11) at (2,0.8){};
\node[vertex](v21) at (2,1.2){};
\node[vertex](v31) at (4,1){};
\node[vertex](v41) at (6,1.2){};
\node[vertex](v51) at (6,0.8){};
\node[vertex](v61) at (8,1.2){};
\node[vertex](v71) at (8,0.8){};

\node[vertex](v12) at (2,3.8){};
\node[vertex](v22) at (2,4.2){};
\node[vertex](v32) at (4,4){};
\node[vertex](v42) at (6,4.2){};
\node[vertex](v52) at (6,3.8){};
\node[vertex](v62) at (8,4.2){};
\node[vertex](v72) at (8,3.8){};

\node[vertexsetb](L12a) at (2,4){};
\node[vertexsetb](L22a) at (4,4){};
\node[vertexsetb](L32a) at (6,4){};
\node[vertexsetb](L42a) at (8,4){};

\node[vertexsetred](L11a1) at (3,2.5){};
\node[vertexsetred](L21a1) at (5,2.5){};
\node[vertexsetred](L31a1) at (7,2.5){};

\node[vertexsetred](L11a1) at (3,5.5){};
\node[vertexsetred](L21a1) at (5,5.5){};
\node[vertexsetred](L31a1) at (7,5.5){};

\node[vertex](u1) at (3,2.5){};
\node[vertex](u2) at (5,2.3){};
\node[vertex](u3) at (5,2.7){};

\node[vertex](u21) at (3,5.5){};
\node[vertex](u22) at (5,5.3){};
\node[vertex](u23) at (5,5.7){};

    \node[vertex, label=left : $s$] (s) at (1,1){};
    \node[vertex, label=right:$t$] (t) at (9,5.5){};

    \draw[arc] (s)--(v11);
    \draw[arc] (s)--(v21);
   % \draw[bluearc] (s)--(v32);
   \begin{scope}[on behind layer]
     \draw[arc] plot [smooth] coordinates {(s) (2.5,3)  (v32)};
    \draw[arc] (s)--(u1);
    \draw[arc]plot [smooth] coordinates { (s)(2.5, 4)(u22)};
	\draw[arc] plot [smooth] coordinates {(s)(2.5, 4.2)(u23)};

    %\draw[bluearc] (v41)--(t);
    \draw[arc] plot [smooth] coordinates {(v41) (7.5,1.8)  (t)};
    %\draw[bluearc] (v51)--(t);
    \draw[arc] plot [smooth] coordinates {(v51) (7.5,1.6)  (t)};
        \end{scope}
    \draw[arc] (v62)--(t);
    \draw[arc] (v72)--(t);
    \draw[arc] (v11)--(v31);
    \draw[arc] (v12)--(v32);
    \draw[arc] (v21)--(v31);
    \draw[arc] (v22)--(v32);
    \draw[arc] (v31)--(v41);
    \draw[arc] (v31)--(v51);
    \draw[arc] (v32)--(v42);
    \draw[arc] (v32)--(v52);
    \draw[arc] (v41)--(v71);
    \draw[arc] (v42)--(v72);
    \draw[arc] (v51)--(v61);
    \draw[arc] (v51)--(v71);
    \draw[arc] (v52)--(v62);
    \draw[arc] (v52)--(v72);
  
    \draw[arc] (u1)--(u3);
    \draw[arc] (u21)--(u23);
%    \draw[redarc] (v11)--(u2);
%    \draw[redarc] (v11)--(u3);
%    \draw[redarc] (v21)--(u2);
%    \draw[redarc] (v21)--(u3);
\begin{scope}[on behind layer]
    \draw[arc] plot [smooth] coordinates {(v11) (3.4,1.4) (4.2,2) (u3)};
    \draw[arc] plot [smooth] coordinates {(v11) (3.4,1.3) (4.4,2) (u2)};
    \draw[arc] plot [smooth] coordinates {(v21) (3.2,1.4) (4.2,2.1)  (u3)};
    \draw[arc] plot [smooth] coordinates {(v21) (3.2,1.3) (4.4,2.1) (u2)};
    \end{scope}
    
    \draw[arc] (u2)--(v62);
    \draw[arc] (u2)--(v72);
    \draw[arc] (u3)--(v62);
    \draw[arc] (u3)--(v72);
    \draw[arc] (u1)--(v42);
    \draw[arc] (u1)--(v52);
    
%    \draw[arc] (v11)--(v42);
%    \draw[arc] (v11)--(v52);
%    \draw[arc] (v21)--(v42);
%    \draw[arc] (v21)--(v52);
\begin{scope}[on behind layer]
    \draw[arc] plot [smooth] coordinates {(v11) (3.6,1.7) (5,3.4) (v42)};
    \draw[arc] plot [smooth] coordinates {(v11) (3.6,1.6) (5.3,3.4) (v52)};
    \draw[arc] plot [smooth] coordinates {(v21) (3.3,1.7) (5,3.5) (v42)};
    \draw[arc] plot [smooth] coordinates {(v21) (3.3,1.6) (5.3,3.5) (v52)};
    %\draw[arc] (v31)--(v62);
    %\draw[arc] (v31)--(v72);
    \draw[arc] plot [smooth] coordinates {(v31) (5.4,1.8)  (v62)};
    \draw[arc] plot [smooth] coordinates {(v31) (5.3,1.6)  (v72)};
        \end{scope}

%\node[font = \small, RoyalBlue, below] at (0,1){$s$};
%\node[font = \small, RoyalBlue, above] at (10,4){$t$};	
\node[font = \scriptsize, below] at (2,0.5){$L_{1,1,\alpha}$};
\node[font = \scriptsize, below] at (4,0.5){$L_{2,1,\alpha}$};
\node[font = \scriptsize, below] at (6,0.5){$L_{3,1,\alpha}$};
\node[font = \scriptsize, below] at (8,0.5){$L_{4,1,\alpha}$};
\node[font = \scriptsize, above] at (2,4.5){$L_{1,2,\alpha}$};
\node[font = \scriptsize, above] at (4,4.5){$L_{2,2,\alpha}$};
\node[font = \scriptsize, above] at (6,4.5){$L_{3,2,\alpha}$};
\node[font = \scriptsize, above] at (8,4.5){$L_{4,2,\alpha}$};
\node[font = \scriptsize, below, red] at (3,2.1){$L_{1,1,\alpha-1}$};
\node[font = \scriptsize, below, red] at (5,2.1){$L_{2,1,\alpha-1}$};
\node[font = \scriptsize, below, red] at (7,2.1){$L_{3,1,\alpha-1}$};
\node[font = \scriptsize, below, red] at (3,5.1){$L_{1,2,\alpha-1}$};
\node[font = \scriptsize, below, red] at (5,5.1){$L_{2,2,\alpha-1}$};
\node[font = \scriptsize, below, red] at (7,5.1){$L_{3,2,\alpha-1}$};
\node[font = \small ] at (9,4){$U_{2,\alpha}$};
\node[font = \small ] at (9,1){$U_{1,\alpha}$};
\node[font = \small, ] at (9,2.5){$U_{1,\alpha-1}$};
\node[font = \small, ] at (9,6){$U_{1,\alpha-1}$};

\node[rahmen] at (5,1){};
\node[rahmen] at (5,4){};
\node[rahmen] at (5,2.5){};
\node[rahmen] at (5,5.5){};
\end{scope}
	
\end{tikzpicture}
\caption{\label{f-blocker-gprime-complete} The graph $G'$ corresponding to the graph $G$ from Figure \ref{f-blocker-example} for $d=2$. }
\end{center}
\end{figure}

We give three properties similar to Properties~\ref{prop-trans-emptypositionspath}-\ref{prop-trans-emptypositionsindset} which will help us to prove the one-to-one correspondence between paths in $G'$ and independent sets in $G$, which we show in Lemma~\ref{lem-blocker-indsetsarepaths}. We omit the proofs due to their similarity to the proofs of Properties~\ref{prop-trans-emptypositionspath}-\ref{prop-trans-emptypositionsindset}.

\begin{property}
\label{prop-blocker-emptypositionspath}
Let $P = s,x_1,\dots, x_h, t$ be an $s$-$t$-path in $G'$. There exist exactly $d-1$ distinct integers in $\intInterval{\alpha}$, say $g_1,\ldots, g_{d-1}$, such that $\pos(x_i)\neq g_1,\ldots,g_{d-1}$ for all $i\in \intInterval{h}$.
For any other integer $g \in \intInterval{\alpha}\setminus \set{g_1,\dots, g_{d-1}}$,  there exists exactly one vertex $x \in P-\{s,t\}$ such that $\pos (x)=g$. 
\end{property}
%\begin{proof}
%Let $P = s,x_1,\dots, x_h, t$ be an $s$-$t$-path in $G'$. 
%By Observation~\ref{obs-trans-construction}, we know that $\pos(x_i) < \pos(x_j)$ for $i<j, i,j \in \intInterval{h}$. 
%Consider an arc $(x_i, x_{i+1})$ in $P$. 
%Let $x_i \in L_{p,\ell, \beta}$ and $x_{i+1}\in L_{p',\ell',\beta'}$ with positions $p,p' \in \intInterval{\alpha}$, levels $\ell,\ell'\in \intInterval{d}$ and $\beta,\beta'\in \set{\alpha-d+1, \dots, \alpha}$.
%We get by construction that $p' = p+ 1 + \ell'-\ell$ and $p'>p,  \ell' \geq \ell$. 
%Hence,  there are $\ell'-\ell$ positions between $x_i$ and $x_{i+1}$, which will not be the position of any vertex in $P$, since $\pos(x_j) < \pos(x_{j'})$ for $j<j', j,j' \in \intInterval{h}$.
%We start at level $1$, end at level $d$ and the increase in levels corresponds to the positions which are not the position of any vertex.
%Thus, we get that there are exactly $d-1$ distinct integers in $\intInterval{\alpha}$, say $g_1,\ldots, g_{d-1}$, such that $\pos(x_i)\neq g_1,\ldots,g_{d-1}$ for all $i\in \intInterval{h}$. 
%Furthermore, it is obvious to see that for any other integer $g\in \intInterval{\alpha}$, there is exactly one vertex in $P- \set{s,t}$ such that $\pos(x) = g$. 
%\end{proof}

\begin{property}
\label{prop-blocker-Pathlength}
Every $s$-$t$-path $P$ in $G'$ has $\alpha-d+3$ vertices.
\end{property}
%\begin{proof}
%From Property \ref{prop-blocker-emptypositionspath}, we know that there are $d-1$ positions that are not the position of any vertex in $P$ and that all other positions are the position of some vertex in $P- \set{s,t}$. 
%Since there are in total $\alpha$ possible positions of which $s$ and $t$ do not use any, there are $\alpha-d+1$ vertices in $P- \set{s,t}$ and hence, $P$ has $\alpha-d+3$ vertices.
%\end{proof}

\begin{property}
\label{prop-blocker-emptypositionsindset}
Let $I = \set{v_1,\dots, v_{\alpha-d+1}}$ be an independent set in $G$. 
There exist exactly $d-1$ distinct integers in $\intInterval{\alpha}$, say $g_1,\ldots, g_{d-1}$ such that $\pos(v_i)\neq g_1,\ldots,g_{d-1}$ for all $i\in \intInterval{\alpha-d+1}$.
\end{property}

Using Property~\ref{prop-blocker-emptypositionspath}-\ref{prop-blocker-emptypositionsindset} we can now proof the one-to-one correspondence between $s$-$t$-paths in $G'$ and independent sets of size $\alpha-d+1$ in $G$.

\begin{lemma}
\label{lem-blocker-indsetsarepaths}
Let $G=(V,E)$ be a co-comparability graph, $\alpha = \alpha(G)$, $d> 0$ an integer and $G'$ constructed as above. 
Every independent set of size $\alpha-d+1$ in $G$ corresponds to an $s$-$t$-path in $G'$ and vice versa.
\end{lemma}
\begin{proof}
We first consider an $s$-$t$-path $P=s,x_1,x_2,\dots, x_{h},t$ in $G'$. 
We know from Property~\ref{prop-blocker-Pathlength} that this path consists of exactly $\alpha-d+3$ vertices, hence $h = \alpha-d+1$. 
Let $v_1,\dots, v_{\alpha-d+1}$ be the vertices in $G$ corresponding to $x_1,x_2,\dots, x_{\alpha-d+1}$.
Notice that for any two sets $L_{p,\ell,\beta}$ and $L_{p',\ell',\beta'}$, there can only be an arc from $x_i\in L_{p,\ell,\beta}$ to $x_j \in L_{p',\ell',\beta'}$ in $G'$ if $p <p',\ell \leq \ell'$. 
Thus, $\pos(v_1)  = \pos(x_1)<\ldots <\pos(x_{\alpha-d+1}) = \pos(v_{\alpha-d+1})$. 

\begin{claim1}
$\set{v_1,\ldots,v_{\alpha-d+1}}$ is an independent set in $G$.
\end{claim1}
\begin{claimproof}
The proof to this claim can be obtained by replacing Property~\ref{prop-trans-orderingtopos} by Lemma~\ref{lemma-blocker-starttoordering} in the proof of Claim~\ref{claim-trans-isindset}. 
\end{claimproof}
%We need to show that for any two vertices $x_i,x_{i+j} \in V(P)$, with $j>1$ and $i,i+j\in[\alpha-d+1]$, their corresponding vertices $v_i, v_{i+j}\in V(G)$ are non-adjacent. 
%From the above we know that $\start(v_i)<\start(v_{i+1})<\ldots <\start(v_{i+j})$. 
%Since we know by construction that $v_i$ and $v_{i+1}$ are non-adjacent, it follows from Lemma~\ref{lemma-blocker-starttoordering} that $v_i \prec v_{i+1} \prec \dots \prec v_{i+j}$. 
%Using Property~\ref{prop-transitive} we get that $v_i$ and $v_{i+j}$ are non-adjacent either.
%We conclude that $\set{v_1,\ldots,v_{\alpha-d+1}}$ is an independent set in $G$. 

We will now prove the converse, that is, we show that an independent set of size $\alpha-d+1$ in~$G$ corresponds to an $s$-$t$-path in $G'$. Let $I = \set{v_1,\dots, v_{\alpha-d+1}}$ be an independent set of size $\alpha-d+1$ in~$G$. 
We assume that we have $v_1\prec v_2 \prec \dots \prec v_{\alpha-d+1}$. 
From Lemma~\ref{lemma-blocker-orderingtostart}, it follows that $\pos(v_1)<\pos(v_{2})<\ldots <\pos(v_{\alpha-d+1})$.

Let $g_1,\dots, g_{d-1}$ be as in Property \ref{prop-blocker-emptypositionsindset}.
For $i \in \intInterval{\alpha-d+1}$, let $x_i\in V'$ be the copy of $v_i$ in $L_{\pos(v_i), \ell,\beta}$, where $\ell = \vert\set{g_k \mid g_k < \pos(v_i), k\in \intInterval{d-1}} \vert + 1$ and $\beta \in \set{\alpha-d+1, \dots, \alpha}$ such that $v_i \in I_\beta$.
We see that $x_i$ exists, since $\ell \leq d$.

To get the existence of a path $P = s,x_1,\dots, x_{\alpha-d+1},t $, it remains to show that the arcs $(s,x_1)$, $(x_i, x_{i+1})$, for $i \in\intInterval{\alpha-d}$, and $(x_{\alpha-d+1},t)$ exist.

\begin{claim1}
The arcs $(s,x_1)$, $(x_i, x_{i+1})$, for $i \in\intInterval{\alpha-d}$, and $(x_{\alpha-d+1},t)$ exist.
\end{claim1}
\begin{claimproof}
This proof is the same as the proof of Claim~\ref{claim-trans-arcsexist}. 
\end{claimproof}
%For the arc $(s,x_1)$ note that $\vert\set{g_k \mid g_k < \pos(v_1), k\in \intInterval{d-1}} \vert + 1 = \pos(v_1)$, and hence the arc exists by definition.
%Let $i \in [\alpha-d]$. 
%Let $\ell_i = \vert\set{g_k \mid g_k < \start(v_i), k\in [d-1]} \vert + 1$ be the level of $x_i$ and let $\ell_{i+1} =\vert\set{g_k \mid g_k < \start(v_{i+1}), k\in [d-1]} \vert + 1$ be the level of $x_{i+1}$.
%Recall that by definition the arc $(x_i,x_{i+1})$ exists if $\start(x_{i+1}) = \start(x_i)+g+1$ and $\ell_{i+1} = \ell_i+g$ for some $g \geq 0$ and $x_i$ and $x_{i+1}$ are non-adjacent for $i\in [\alpha-d]$.
%Note that $\ell_{i+1} - \ell_i$ gives exactly the number of positions between $\start(x_i)$ and $\start(x_{i+1})$ and hence we get $\start(x_{i+1}) = \start(x_i)+\ell_{i+1} - \ell_i+1$.
%Further, we know that $\set{v_i, v_{i+1}}$ are non-adjacent.
%Thus, the arc $(x_i, x_{i+1})$ exists by definition.
%If we consider $x_{\alpha-d+1}$, we get for the level $\ell$ of $x_{\alpha-d+1}$ that $\ell = \vert\set{g_k \mid g_k < \start(v_{\alpha-d+1}), k\in [d-1]} \vert + 1 =  \start(v_{\alpha-d+1}) + d - \alpha$.
%Hence, the arc $(x_{\alpha-d+1}, t)$ exists by definition.
It follows that $P = s,x_1,\dots, x_{\alpha-d+1},t $ is a path in $G'$. This concludes the proof of the lemma.
\end{proof}

Recall that $\pos(v)$ denotes the position of $v$ in a maximum independent set containing $v$, while $\pos_I(v)$ denotes the position of $v$ in a specific, not necessarily maximum independent set $I$. To proof our second main result we need a last property, whose proof is omitted since it is similar to the proof of Property~\ref{prop-trans-levelpos}.

\begin{property}
\label{prop-blocker-levelpos}
Let $G$ be a co-comparability graph with vertex ordering $\prec$, $G'$ as constructed above.
Let $I_1,I_2$ be independent sets of size $\alpha-d+1$ in $G$ and let $P_1,P_2$ be their corresponding paths in $G'$.
Suppose there is $v \in I_1 \cap I_2$ such that the corresponding vertices $x_1\in P_1,  x_2\in P_2$ are different.  Assume without loss of generality that $\level(x_1) < \level(x_2)$.
Then,  $\pos_{I_1}(v)> \pos_{I_2}(v)$.
\end{property}
%\begin{proof}
%Let $g_1^1,\dots, g_{d-1}^1 \in \intInterval{\alpha}$ be the integers such that $\pos(x) \neq g_1,\dots, g_{d-1}$ for any $x \in P_1$, which exist by Property~\ref{prop-blocker-emptypositionspath}. Let $g_1^2, \dots, g_{d-1}^2$ be this set of integers corresponding to $P_2$.
%Recall that $\level(x_1) = |\set{g_k^1 \mid g_k^1 < \pos(v)|k \in \intInterval{d-1}}|+1$ and $\level(x_2) = |\set{g_k^2 \mid g_k^2 < \pos(v)|k \in \intInterval{d-1}}|+1$.
%Since $\pos_{I_1}(v) = \pos(v) - \level(x_1)$ and $\pos_{I_2}(v) = \pos(v) - \level(x_2)$ and $\level(x_1) < \level(x_2)$, we know that  $\pos_{I_1}(v) > \pos_{I_2}(v)$. 
%\end{proof}

\begin{theorem}
\label{theo-blockeralpha}
\del$(\alpha)$ is polynomial-time solvable for co-comparability graphs.
\end{theorem}
\begin{proof}

Let $G=(V,E)$ be a co-comparability graph and let $(G,d,k)$ be an instance of \del($\alpha$). We construct the graph $G'=(V',A')$ as described above. We will show that $(G,d,k)$ is a \yes-instance of \del($\alpha$) if and only if $(G',k)$ is a \yes-instance of \mincut.

Let $(G',k)$ be a \yes-instance of \mincut \ and let $C$ be an $s$-$t$-cut of $G'$ of size at most $k$. We want to prove that $(G,d,k)$ is a \yes-instance of \del($\alpha$).

For every vertex in the cut $C\subseteq V'$, we add the corresponding vertex in $G$ to a set $S$. 
We assume for a contradiction that there is an independent set $I$ in $G- S$ of size $\alpha-d+1$. 
By Lemma \ref{lem-blocker-indsetsarepaths}, we know that there is an $s$-$t$-path $P$ in $G'$ representing $I$. 
Since $I \subseteq V\setminus S$, we get that $P\cap C = \varnothing$ and hence, we can find an $s$-$t$-path in $G'- C$, a contradiction.
Thus, such an independent set $I$ does not exist, and so by Observation~\ref{obs-blocker-intersection}, we deduce that $(G,d,k)$ is a \yes-instance of \del($\alpha$).

Let now $(G,d,k)$ be a \yes-instance of \del($\alpha$). 
We want to show that $(G',k)$ is a \yes-instance of \mincut. 
Let $S\subseteq V$ with $|S|\leq k$ be such a $d$-deletion blocker of $G$. 
We may assume that $S$ is minimal.

We iteratively construct a set $C$ using again Algorithm~\ref{AlgoCutconstruction}, with the difference that we get as input the graph $G'$ and $d$-blocker $S$ in $G$ instead of a $d$-transversal. We will prove that $C$ is an $s$-$t$-cut in $G'$ with $|S| = |C|$. 
For each vertex in $S$, the algorithm chooses the corresponding vertex in $G'$ belonging to the lowest level such that there is an $s$-$t$-path in $G'- C$ containing this vertex and then adds it to $C$. 
We will find such a vertex for every vertex in $S$, since otherwise $S$ would not be minimal. 
Hence, it is clear that $|S| = |C|$.
%\begin{algorithm}
%\caption{}
%\label{algo-blocker-Cutconstruction}
%\begin{algorithmic}[2]
%\Require A graph $G'$ constructed from a co-comparability graph $G$, a minimal vertex set $S\subseteq V(G)$ such that in $G-S$ there is no independent set of size $\alpha-d+1$.
%\Ensure An $s$-$t$-cut $C \subseteq V(G')$ with $|S| = |C|$.
%\State Let $S = \set{u_1,\dots, u_{|S|}}, u_i \prec u_j$ for $i < j$, $i,j \in \intInterval{|S|}$.
%\State Let $C = \varnothing$.
%\For {$i $ from $1\to |S|$}
%	\State Let $u = u_i$.
%	\State Let $y_1,\dots, y_{d} \in V'$ be the vertices corresponding to $u$, sorted by increasing level.
%	\For{$j$ from $1 \to d$}
%		\If {$\exists$ $s$-$t$-path in $G'- C$ containing $y_j$}
%			\State $C = C \cup y_j$.
%			 \State \textbf{break}
%		\Else
%			\State \textbf{continue}
%		\EndIf
%	\EndFor
%\EndFor
%\end{algorithmic}
%\end{algorithm}

To prove that $C$, which is constructed by application of Algorithm~\ref{AlgoCutconstruction}, is indeed an $s$-$t$-cut in~$G'$, we assume for a contradiction that there exists an $s$-$t$-path $P_1$ in $G'- C$. 
Let $I_1\subseteq V$ be the independent set in $G$ corresponding to $P_1$, which exists by Lemma \ref{lem-blocker-indsetsarepaths}. 
Since $S$ is a $d$-deletion blocker of $G$, we know that $S\cap I_1 \neq \varnothing$. 
Let $v \in S\cap I_1$ and if $|S\cap I_1| > 1$, we choose the rightmost vertex according to $\prec$ in $S\cap I_1$ for $v$. 
Let $y_j\in V'$, $j\in \intInterval{d}$, be the copy of $v$ in $P_1$. 
Since $P_1$ is a path in $G'- C$, we know that $y_j \not\in C$. 
Hence, there is some other copy of $v$, say $y_i$, $i \in \intInterval{d}$,  such that Algorithm~\ref{AlgoCutconstruction} added $y_i$ to $C$. 
Due to the procedure we used in Algorithm~\ref{AlgoCutconstruction} to choose the vertices in $C$, we have $i < j$.

Since $y_i$ was added to $C$, there exists a path $P_2$ in $G'$ containing $y_i$. 
Let $I_2$ be the independent set in $G$ corresponding to $P_2$, which exists by Lemma~\ref{lem-blocker-indsetsarepaths}.
From Property~\ref{prop-blocker-levelpos} we know that $\pos_{I_1}(v) < \pos_{I_2}(v)$. 
Let $J_1 \subseteq I_1$ be the set of the $\alpha-d+1-\pos_{I_1}(v)$ rightmost vertices in $I_1$, i.e.\,those vertices $u$ in $I_1$ such that $v \prec u$ (see Figure~\ref{f-blocker-compindset}a)).  
We know that $J_1 \cap S = \varnothing$ by the choice of $v$. 
Let $J_2$ be the set of the $\pos_{I_2}(v)-1$ leftmost vertices from $I_2$ (see Figure~\ref{f-blocker-compindset}a)), i.e.\,those vertices $u$ in $I_2$ with $u \prec v$.
It follows from Property~\ref{prop-transitive} that $J_2 \prec v \prec J_1$ is an independent set. 
Since $\pos_{I_1}(v) < \pos_{I_2}(v)$, we get that $|J_2\cup J_1| \geq \alpha-d+1$.
Hence, $J_2 \cup J_1$ either is an independent set of size at least $\alpha-d+1$ in $G- S$ or $J_2 \cap S \neq \varnothing$. 
In the first case, we directly get a contradiction to our assumption that $S$ is a $d$-deletion blocker in $G$.
So, we may assume that $J_2 \cap S \neq \varnothing$.

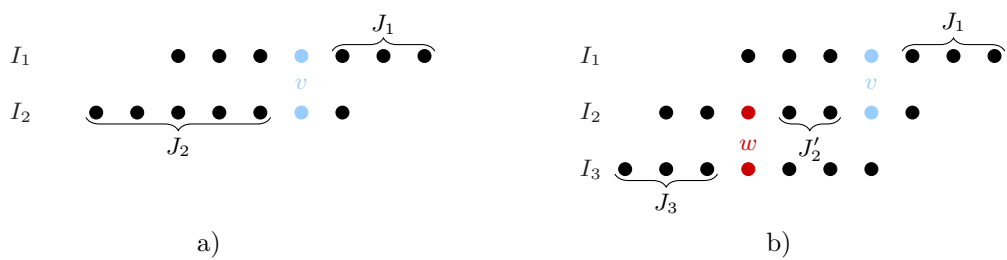
\begin{figure}
\centering
\begin{tikzpicture}

	\tikzstyle{interval} = [thick]
	\tikzstyle{intervald} = [thick, dotted]
	\tikzstyle{intervalv} = [thick, RoyalBlue]
	\tikzstyle{intervalu} = [thick, ForestGreen]
	\tikzstyle{intervalr} = [thick, Red]
	\tikzstyle{interval1d} = [thick, ForestGreen, dotted]
	\tikzstyle{indsetJ} = [thick, Red]
	\def\ha{3}
	\def\hb{2.25}
	\def\hc{1.5}
	
	\begin{scope}[ xscale = 0.9]

	\node[vertex](a) at (0.8,\ha){};	
	\node[vertex](a) at (1.4,\ha){};	
	\node[vertex](a) at (2,\ha){};	
	\node[bvertex](a) at (2.6,\ha){};	
	\node[vertex](a) at (3.2,\ha){};	
	\node[vertex](a) at (3.8,\ha){};	
	\node[vertex](a) at (4.4,\ha){};		
	
	\node[vertex](a) at (-0.4,\hb){};
	\node[vertex](a) at (0.2,\hb){};	
	\node[vertex](a) at (0.8,\hb){};	
	\node[vertex](a) at (1.4,\hb){};	
	\node[vertex](a) at (2,\hb){};	
	\node[bvertex](a) at (2.6,\hb){};	
	\node[vertex](a) at (3.2,\hb){};

	\node[font = \small, lightblue] at (2.6,2.625){$v$};
	
	\node[font = \small, Black] at (-1.5,\ha){$I_1$};
	\node[font = \small, Black] at (-1.5,\hb){$I_2$};

	\draw [decorate,decoration={brace, mirror,amplitude=5pt},xshift=0pt,yshift=11pt] (-0.55,1.8) -- (2.15,1.8) node [black,midway,yshift=-4pt, below] {\small $J_2$};
	\draw [decorate,decoration={brace,,amplitude=5pt},xshift=0pt,yshift=-11pt] (3.05,3.45) -- (4.55,3.45) node [black,midway,yshift=3pt, above] {\small $J_1$};

	\node[] at (1.25,0.5){a)};
	\end{scope}
	
	\begin{scope}[shift = {(7.5,0)}, xscale = 0.9]

	\node[vertex](a) at (0.8,\ha){};	
	\node[vertex](a) at (1.4,\ha){};	
	\node[vertex](a) at (2,\ha){};	
	\node[bvertex](a) at (2.6,\ha){};	
	\node[vertex](a) at (3.2,\ha){};	
	\node[vertex](a) at (3.8,\ha){};	
	\node[vertex](a) at (4.4,\ha){};		
	
	\node[vertex](a) at (-0.4,\hb){};
	\node[vertex](a) at (0.2,\hb){};	
	\node[rvertex](a) at (0.8,\hb){};	
	\node[vertex](a) at (1.4,\hb){};	
	\node[vertex](a) at (2,\hb){};	
	\node[bvertex](a) at (2.6,\hb){};	
	\node[vertex](a) at (3.2,\hb){};

	\node[vertex](a) at (-1,\hc){};
	\node[vertex](a) at (-0.4,\hc){};	
	\node[vertex](a) at (0.2,\hc){};	
	\node[rvertex](a) at (0.8,\hc){};	
	\node[vertex](a) at (1.4,\hc){};	
	\node[vertex](a) at (2,\hc){};	
	\node[vertex](a) at (2.6,\hc){};

	\node[font = \small, lightblue] at (2.6,2.625){$v$};
	\node[font = \small, nicered] at (0.8,1.825){$w$};
	
	\node[font = \small, Black] at (-1.5,\ha){$I_1$};
	\node[font = \small, Black] at (-1.5,\hb){$I_2$};
	\node[font = \small, Black] at (-1.5,\hc){$I_3$};

	\draw [decorate,decoration={brace, mirror,amplitude=5pt},xshift=0pt,yshift=11pt] (-1.15,1.05) -- (0.35,1.05) node [black,midway,yshift=-4pt, below] {\small $J_3$};
		\draw [decorate,decoration={brace, mirror,amplitude=5pt},xshift=0pt,yshift=11pt] (1.25,1.8) -- (2.15,1.8) node [black,midway,yshift=-4pt, below] {\small $J_2'$};
		\draw [decorate,decoration={brace,,amplitude=5pt},xshift=0pt,yshift=-11pt] (3.05,3.45) -- (4.55,3.45) node [black,midway,yshift=4pt, above] {\small $J_1$};
		
		\node[] at (1.25,0.5){b)};

	\end{scope}

\end{tikzpicture}
%\begin{tikzpicture}
%
%
%	\tikzstyle{rahmenb}=[rounded corners = 5pt,draw, dashed, minimum width = 310pt, minimum height = 30pt]
%	\tikzstyle{rahmeng}=[rounded corners = 5pt,draw=ForestGreen, dashed, minimum width = 142.5pt, minimum height = 30pt]
%	\tikzstyle{vertexsetb}=[rounded corners = 4pt,draw, minimum width = 10pt, minimum height = 25pt]
%	\tikzstyle{vertexsetg}=[rounded corners = 4pt,draw=ForestGreen, minimum width = 10pt, minimum height = 25pt,]
%	\tikzstyle{interval} = [thick]
%	\tikzstyle{intervalv} = [thick, RoyalBlue]
%	\tikzstyle{intervalu} = [thick, ForestGreen]
%	
%	\node[vertex](s) at (0,0){};
%	\node[vertex](t) at (10,5){};
%	\node[rvertex](u1) at (3,2){};
%	\node[rvertex](u2) at (3,1){};
%	\node[bvertex](vi) at (6,3){};
%	\node[bvertex](vj) at (6,4){};
%	
%	\begin{scope}[on behind layer]
%		\draw[edge] plot [smooth] coordinates {(s) (2.5,2.5) (vj)  (t)};
%		\draw[edge] plot [smooth] coordinates {(s)(u1) (vi)  (t)};
%		\draw[edge] plot [smooth] coordinates {(s)(u2) (7.5,2.5)  (t)};
%	\end{scope}
%	
%	
%	\node[font = \small, Black, below] at (0,0){$s$};
%\node[font = \small, Black, above] at (10,5){$t$};	
%\node[font = \small, Black,] at (8,4.8){$P$};
%\node[font = \small, Black] at (8,3.6){$P_1$};	
%\node[font = \small, Black,] at (8,2.5){$P_2$};
%
%\node[font = \small, Black, below] at (6,3){$v_i$};
%\node[font = \small, Black, above] at (6,4){$v_j$};	
%\node[font = \small, Black, below] at (3,1){$u_g$};
%\node[font = \small, Black, above] at (3,2){$u_h$};	
%
%\node[] at (5,-0.5){c)};
%	
%\end{tikzpicture}
\caption{\label{f-blocker-compindset}The different independent sets considered in the proof. }
\end{figure}

Let $w\in J_2 \cap S$, and if $|J_2 \cap S|>1$, we take the rightmost vertex $w$ in $J_2 \cap S$ with respect to $\prec$ such that $w \prec v$. 
Let $y_h$, $h\in \intInterval{d}$ be the vertex in $P_2$ corresponding to $w$. 
Then $y_h \not\in C$, since otherwise Algorithm~\ref{AlgoCutconstruction} would not have added $y_h$ to $C$. 
Thus, as before, there exists some vertex $y_g$, $g \in \intInterval{d}$, with $g < h$, such that $y_g$ corresponds to $w$ and $y_g \in C$.
Therefore, there exists an $s$-$t$-path $P_3$ in $G'$ containing $y_g$ with corresponding independent set $I_3$.
$I_3$ contains more vertices to the left of $w$ than $I_2$, since by Property~\ref{prop-blocker-levelpos} we have that $\pos_{I_2}(w) < \pos_{I_3}(w)$ (see Figure \ref{f-blocker-compindset}b)). 
Remember that $J_2 \cup J_1 $ is an independent set.
We consider the set $J_2' = \set{u \in I_2|\pos_{I_2}(w) < \pos_{I_2}(u) < \pos_{I_2}(v)} = \set{u \in I_2| w \prec u \prec v}$ which is a subset of the set $J_2$. 
Thus, $J_2' \cup J_1$ is still an independent set.
We define a new set $J_3 = \set{u \in I_3|\pos_{I_3}(u) < \pos_{I_3}(w)}\subseteq I_3$, which is an independent set, since $I_3$ is an independent set. 
From the fact that $J_3$, respectively $J_2'\cup J_1$, contains only vertices on the left, respectively right, of $w$ and that both are independent to $w$ we get that $J_3 \cup J_2' \cup J_1$ is an independent set.
Since $\pos_{I_1}(v) < \pos_{I_2}(v)$ and $\pos_{I_2}(w) < \pos_{I_3}(w)$, we get that $|J_3 \cup J_2' \cup J_1| \geq \alpha-d+1$. 
Hence, in $G- S$ we get again either an independent set $J_3 \cup J_2' \cup J_1$ of size at least $\alpha-d+1$ or $J_3\cap S\neq \varnothing$.

By repeatedly using these arguments, we can always find a new vertex in $S$. But since $S$ is finite, this case can not always occur. Hence, we will necessarily get a contradiction and thus, there is no $s$-$t$-path $P$ in $G'-C$.  So we conclude that $C$ is an $s$-$t$-cut in $G'$.

Let us now consider the complexity of our algorithm. From~\cite{tarjanmaxflow} we know that for a graph with $n$ vertices we can solve \mincut \ in $\mathcal{O}(n^3)$. Since the graph $G'$ has $\mathcal{O}(d|V|)$ vertices, computing a \mincut \ in $G'$ can be done in time $\mathcal{O}(d^3|V|^3)$.
We still need to consider the time we need to construct $G'$. Using Lemma~\ref{lem-blocker-computepartition}, we know that we can find the partition of $\I$ into sets $L_{i}$ in $\mathcal{O}(|V|^2)$. For every pair of vertices in $G'$, we can check in $\mathcal{O}(|V|)$ if we introduce an arc between them in~$G'$. Hence, $G'$ can be constructed in $\mathcal{O}(d^2|V|^3)$.
We conclude that \del($\alpha$) can be solved in time $\mathcal{O}(d^3|V|^3)$.
\end{proof}

\section{Conclusion}
\label{sec:conclusion}
In this paper, we showed that \del($\alpha$) and \trans($\alpha$) are polynomial-time solvable in the class of co-comparability graphs by reducing them to the well-known \mincut\ problem. This generalises results of \cite{CKL01} and \cite{HLW23}. We believe that our approach (reduction to the \mincut\ problem) can also be used to solve the weighted version of our problems in polynomial time, i.e., where we consider maximum weighted independent sets. The same holds for \del($\gamma$) and \trans($\gamma$), i.e., the blocker and the transversal problems with respect to minimum dominating sets. We leave this as open questions.

\medskip
\noindent
\textit{Acknowledgements.}
The first author thanks Jannik Olbrich for helpful discussions.

\bibliography{ref}

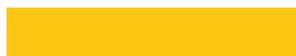
\end{document}